\documentclass[12pt]{amsart}

\usepackage[utf8]{inputenc}						% French-friendly encoding (latin1 seems to cause problems)
\usepackage{amsmath,amsfonts,amssymb,amsthm}	% standard math packages

\usepackage{geometry}		% page layout
\usepackage{mathtools}		% extra math commands
\usepackage{hyperref}		% hyperlinks
\usepackage{setspace}		% interline spacing
\usepackage{cite}			% ordering of reference numbers
\usepackage{color}			% colors

\definecolor{midpurple}{rgb}{0.6,0.2,0.4}
\hypersetup{colorlinks=true,citecolor=blue,linkcolor=blue,urlcolor=midpurple}

\numberwithin{equation}{section}			% equation numbered as (section number:number)
\newtheorem{theorem}{Theorem}[section]			% theorems numbered as 'section':'theorem'
\newtheorem{proposition}[theorem]{Proposition}	% propositions, lemmas, etc. use and increment the 'theorem' counter
\newtheorem{lemma}[theorem]{Lemma}
\newtheorem{corollary}[theorem]{Corollary}

\newtheorem{remark}[theorem]{Remark}
\newtheorem{definition}[theorem]{Definition}

\newtheorem*{theorem*}{Theorem}
\newtheorem*{proposition*}{Proposition}
\newtheorem*{lemma*}{Lemma}
\newtheorem*{corollary*}{Corollary}

\newtheorem*{conjecture*}{Conjecture}
\newtheorem*{question*}{Question}
\newtheorem*{remark*}{Remark}
\newtheorem*{definition*}{Definition}

\renewcommand{\leq}{\leqslant}	% nice =<
\renewcommand{\geq}{\geqslant}	% nice >=

\newcommand{\N}{\mathbb{N}}		% positive integers
\newcommand{\Z}{\mathbb{Z}}		% relative integers
\newcommand{\Q}{\mathbb{Q}}		% rationals
\newcommand{\R}{\mathbb{R}}		% reals
\newcommand{\C}{\mathbb{C}}		% complex numbers
\newcommand{\T}{\mathbb{T}}		% torus
\newcommand{\eps}{\varepsilon}			% nice epsilon
\newcommand{\wt}{\widetilde}		% wide tilde
\newcommand{\wh}{\widehat}			% Fourier hat
\newcommand{\E}{\mathbb{E}}			% expectancy operator
\newcommand{\transp}{\mathsf{T}}						% transpose, English postscript style

\newcommand{\dx}{\mathrm{d}x}
\newcommand{\dm}{\mathrm{d}m}

\newcommand{\dbeta}{\mathrm{d}\beta}
\newcommand{\deta}{\mathrm{d}\eta}

\newcommand{\dsigma}{\mathrm{d}\sigma}
\newcommand{\dtheta}{\mathrm{d}\theta}
\newcommand{\dxi}{\mathrm{d}\xi}
\newcommand{\dSigma}{\mathrm{d}\Sigma}

\newcommand{\bfa}{\mathbf{a}}
\newcommand{\bfb}{\mathbf{b}}

\newcommand{\bfe}{\mathbf{e}}

\newcommand{\bfj}{\mathbf{j}}
\newcommand{\bfk}{\mathbf{k}}

\newcommand{\bfm}{\mathbf{m}}
\newcommand{\bfn}{\mathbf{n}}

\newcommand{\bfu}{\mathbf{u}}
\newcommand{\bfv}{\mathbf{v}}
\newcommand{\bfw}{\mathbf{w}}
\newcommand{\bfx}{\mathbf{x}}
\newcommand{\bfy}{\mathbf{y}}
\newcommand{\bfz}{\mathbf{z}}

\newcommand{\bfM}{\mathbf{M}}

\newcommand{\bfP}{\mathbf{P}}

\newcommand{\bfalpha}{\boldsymbol{\alpha}}	
\newcommand{\bfbeta}{\boldsymbol{\beta}}	
	
\newcommand{\bflambda}{\boldsymbol{\lambda}}	
\newcommand{\bfmu}{\boldsymbol{\mu}}
\newcommand{\bftheta}{\boldsymbol{\theta}}	

\newcommand{\bfxi}{\boldsymbol{\xi}}

\newcommand{\dbfx}{\mathrm{d}\mathbf{x}}
\newcommand{\dbfalpha}{\mathrm{d}\boldsymbol{\alpha}}
\newcommand{\dbfbeta}{\mathrm{d}\boldsymbol{\beta}}	

\newcommand{\dbfxi}{\mathrm{d}\boldsymbol{\xi}}

\newcommand{\frakM}{\mathfrak{M}}		% major arcs
\newcommand{\frakm}{\mathfrak{m}}		% minor arcs
\newcommand{\frakU}{\mathfrak{U}}		% fundamental domain
\newcommand{\frakS}{\mathfrak{S}}		% singular series
\newcommand{\frakJ}{\mathfrak{J}}		% singular integral
\newcommand{\calB}{\mathcal{B}}			% factor
\newcommand{\calN}{\mathcal{N}}			% number of solutions
\newcommand{\calP}{\mathcal{P}}			% property

\newcommand{\ZM}{\Z_{\overline{\mathbf{M}}}}										% cyclic group \prod \Z/M^{k_i}\Z
\DeclareMathOperator{\diam}{diam}													% diameter
\newcommand{\EB}[1]{\mathbb{E}\big[ #1 | \mathcal{B} \big]}							% conditional expectation

\begin{document}

\title{Additive equations in dense variables
via truncated restriction estimates}

\author{Kevin Henriot}

\date{}
\begin{abstract}
We study additive equations of the form
$\sum_{i=1}^s \lambda_i \bfP(\bfn_i) = 0$
in variables $\bfn_i \in \Z^d$, where the $\lambda_i$ 
are nonzero integers summing up to zero and $\bfP = (P_1,\dots,P_r)$ 
is a system of homogeneous polynomials
making the equation is translation-invariant.
We investigate the solvability of this equation in subsets
of density $(\log N)^{-c(\bfP,\bflambda)}$ of a large box $[N]^d$,
via the energy increment method.
We obtain positive results for roughly the number of variables
currently needed to derive a count of solutions in the complete box $[N]^d$,
for the multidimensional systems of large degree 
studied by Parsell, Prendiville and Wooley.
Appealing to estimates from the decoupling theory
of Bourgain, Demeter and Guth, we also treat the cases
of the monomial curve $\bfP = (x,\dots,x^k)$
and the parabola $\bfP=(\bfx,|\bfx|^2)$,
for a number of variables close to or equal
to the limit of the circle method.
\end{abstract}

\maketitle

% XSTART

\section{Introduction}
\label{sec:intro}

We are interested in solving
additive diophantine equations in
variables belonging to a thin subset of
a box $[N]^d$, for a large integer $N \geq 2$.
More precisely, we consider a system of $r$ homogeneous integer polynomials 
$\bfP = (P_1,\dots,P_r)$ in $d$ variables,
with each $P_i$ of degree $k_i \geq 1$.
Borrowing terminology from Parsell et al.~\cite{PPW:Multidim},
we call $d = d(\bfP)$ the dimension of the system $\bfP$
when each variable $x_i$, $1 \leq i \leq d$ appears in 
a monomial with nonzero coefficient in 
at least one of the polynomials $P_1,\dots,P_r$.
We define the degree of $\bfP$ as $k = k(\bfP) = \max_i k_i$,
and its weight as $K = K(\bfP) = \sum_i k_i$.
Furthermore, we say that the system is reduced when the polynomials
$P_i$ are linearly independent, in which case we call $r = r(\bfP)$
the rank of the system.
We also fix coefficients $\lambda_1,\dots,\lambda_s \in \Z \smallsetminus \{ 0 \}$
%such that $\lambda_1 + \dotsb + \lambda_s = 0$,
and study the system of $r$ equations given by
\begin{align}
\label{eq:intro:SystPols}
	\lambda_1 \bfP( \bfx_1 ) + \dotsb + \lambda_s \bfP( \bfx_s ) = 0,
\end{align}
with variables $\bfx_1,\dots,\bfx_s \in \Z^d$.
%where the variables $\bfx_1,\dots,\bfx_s$ are 
%constrained to lie in a subset $A$ of $[N]^d$ of density $\delta$. 
In order to solve this system in variables belonging
to subsets of $\Z^d$, we make
the additional assumption
that~\eqref{eq:intro:SystPols} is translation-invariant\footnote{
By this we mean that
when $(\bfx_1,\dots,\bfx_s)$ 
is a solution of~\eqref{eq:intro:SystPols},
so is $(\bfx_1 + \bfu,\dots,\bfx_s + \bfu)$ for every $\bfu \in \Z^d$.},
which imposes the condition $\lambda_1 + \dotsb + \lambda_s = 0$
that we assume from now on.
Our assumption of homogeneity also 
guarantees that~\eqref{eq:intro:SystPols} is dilation-invariant.
Depending on the equation under study, one also typically
defines a notion of non-trivial solution which, at the very least,
excludes the trivial diagonal solutions 
$\bfx_1 = \dots = \bfx_s$.

Via Taylor expansions, one way to obtain translation-invariance in~\eqref{eq:intro:SystPols} 
is to pick a linearly independent subset $\bfP$ 
of the set of all partial derivatives of a given family of
polynomials $Q_1,\dots,Q_h \in \Z[x_1,\dots,x_d]$,
in which case we say that $\bfP$ is the seed system generated
by the seed polynomials $Q_1,\dots,Q_h$.
We also recall a more general definition of Parsell et al.~\cite[Section~2]{PPW:Multidim}:
we say that the system $\bfP$ is translation-dilation invariant if
there exists a lower unitriangular matrix $C(\bfxi)$
and a vector $c_0(\bfxi)$ whose entries are integer polynomials in $\bfxi$
such that
\begin{align*}
	&\phantom{(\bfx,\bfxi \in \Z^d)} &
	\bfP( \bfx + \bfxi ) &= c_0(\bfxi) + C(\bfxi) \bfP( \bfx )
	&&(\bfx,\bfxi \in \Z^d).
\end{align*}
It can be verified that this class of systems of polynomials
contains the seed systems, and that it ensures again
translation-dilation invariance in the equation~\eqref{eq:intro:SystPols}.

A classical question in additive combinatorics is to bound from
below the lowest admissible density $\delta = \delta(N)$ such that
any subset $A$ of $[N]^d$ of density at least $\delta$
contains a non-trivial solution to~\eqref{eq:intro:SystPols},
as $N$ tends to infinity.
When specializing to the equation $x_1 + x_3 = 2x_2$ detecting
three-term arithmetic progressions,
this covers the classical setting of Roth's theorem~\cite{Roth:Roth},
which says that the equation has a solution with all $x_i$ distinct
in any subset of $[N]$ of density at least $(\log\log N)^{-c}$.
A subsequent argument of Szemerédi~\cite{Szemeredi:Roth} 
and Heath-Brown~\cite{HB:Roth} lowered the admissible density
to $(\log N)^{-c}$, for a small constant $c > 0$.
A new framework was developed by Bourgain~\cite{Bourgain:Roth} to obtain
the exponent $c = 1/2 - \eps$, but in this work we only rely
on the Heath-Brown-Szemerédi machinery.

The study of this question in cases of higher degree or dimension
has generated a fair amount of interest recently.
The work of Smith~\cite{Smith:Diag} and Keil~\cite{Keil:Diag}
concerned the one-dimensional quadratic case $\bfP = (x,x^2)$.
Smith~\cite{Smith:VinoSyst} has studied the degree-$k$ case $\bfP = (x,\dots,x^k)$,
and Prendiville~\cite{Prendiville:BinaryForms}
has investigated the two-dimensional setting where
$\bfP$ is given by a binary form and its derivatives.
Prendiville's result was later generalized in work of
Parsell et al.~\cite{PPW:Multidim} to the class of all 
translation-dilation invariant systems of polynomials.
In these references,
%~\cite{Smith:Diag,Smith:VinoSyst,Keil:Diag,Prendiville:BinaryForms,PPW:Multidim}, 
doubly logarithmic bounds of the shape $(\log\log N)^{-c(s)}$
were obtained via the method of Roth~\cite{Roth:Roth},
for a number of variables sufficient
to count the number of solutions to~\eqref{eq:intro:SystPols} 
in $[N]^d$ by the circle method.
In our previous work~\cite{me:logkeil}, we obtained logarithmic bounds
of the shape $(\log N)^{-c(s,\bflambda)}$ 
for the case $\bfP = (x,x^2)$, by adapting the 
Heath-Brown-Szemerédi method~\cite{HB:Roth,Szemeredi:Roth}.
The purpose of this work is to generalize this
result to cases of larger degree or dimension.

The discussion of our main theorem requires a little more context,
but we can start by stating a representative result.
Following Parsell et al.~\cite{PPW:Multidim},
we say that $(\bfx_1,\dots,\bfx_s) \in (\Z^d)^s$ is 
a projected solution of~\eqref{eq:intro:SystPols}
when all of the $\bfx_i$ belong to a proper affine
subspace of $\Q^d$ ;
in dimension one this is equivalent to $\bfx_1 = \dots = \bfx_s$.
We say that $\bfx$ is a subset-sum solution 
when there exists a partition
$[s] = E_1 \bigsqcup \dotsb \bigsqcup E_\ell$ with $\ell \geq 2$
such that, for all $j \in [\ell]$,
$\sum_{i \in E_j} \lambda_i = 0$
and $\sum_{i \in E_j} \lambda_i \bfP(\bfx_i) = 0$.
This second definition is meant to exclude the obvious solutions
obtained by setting the $(\bfx_i)_{i \in E_j}$
to be equal for each $j \in [\ell]$.
Note that the space of projected solutions, and that
of subset-sum solutions are translation-dilation invariant\footnote{
That is, they are invariant under
translations $(\bfx_j)_{1 \leq j \leq s} \mapsto (\bfx_j + \bfu)_{1 \leq j \leq s}$, $\bfu \in \Q^d$
and dilations $(\bfx_j)_{1 \leq j \leq s} \mapsto \gamma (\bfx_j)_{1 \leq j \leq s}$, $\gamma \in \Q$.}.

\begin{theorem}[Additive equations in subsets of monomial surfaces]
\label{thm:intro:SystMonomMultidim}
Let $k \geq 2$, $d \geq 1$, $s \geq 1$ and $\lambda_1, \dots, \lambda_s \in \Z \smallsetminus \{0\}$
be such that $\lambda_1 + \dotsb + \lambda_s = 0$.
Suppose that 
\begin{align*}
%\label{eq:intro:ParsellPols}
	\bfP = ( x_1^{j_1} \cdots x_d^{j_d} ,\ 1 \leq j_1 + \dotsb + j_d \leq k )
\end{align*}
%or $\bfP = (x_1^{j_1} \cdots x_d^{j_d} ,\ \bfj \in [0,k] \smallsetminus \{0\})$,
and let $r$ denote the rank of $\bfP$.
Suppose also that the system of equations~\eqref{eq:intro:SystPols}
possesses nonsingular real and $p$-adic solutions for every prime $p$.
When $s \geq 2r(k+1) + 1$, there exists a constant $c(d,k,\bflambda) > 0$
such that every subset of $[N]^d$ of density at least\footnote{
Note that this forces $N$ to be larger than a certain constant
depending on $\bfP$ and $\bflambda$.
The constant $c(d,k,\bflambda)$ absorbs dependencies on $s$,
considered as the dimension of the vector $(\lambda_1,\dots,\lambda_s)$.}
$2(\log N)^{-c(d,k,\bflambda)}$
contains a solution to the system of equations~\eqref{eq:intro:SystPols},
which is neither a projected nor a subset-sum solution.
\end{theorem}

Note that the system of polynomials $(\bfx^{\bfj}, 1 \leq |\bfj| \leq k)$
is generated by the seed polynomials $(\bfx^{\bfj},|\bfj|=k)$.
For that system,
% of polynomials~\eqref{eq:intro:ParsellPols}, 
the estimates of Parsell et al.~\cite{PPW:Multidim}
for multidimensional Vinogradov mean values
allow for a circle method treatment 
of the equation~\eqref{eq:intro:SystPols}
in the same range $s \geq 2r(k+1) + 1$,
and this is a substantial input in our proof.
An important aspect of our approach, however, is that we need 
little number theoretic information beyond mean value
estimates to handle dense variables, 
and in the case of the above theorem
the additional requirements consist only 
in simple bounds for local
multidimensional exponential sums.
%due to Arkhipov et al.~\cite{ACK:Book}.

We now discuss in some depth the Fourier-analytic estimates
involved in the treatment equation~\eqref{eq:intro:SystPols} 
in dense variables, in order to motivate our main result.
We define the weighted and unweighted exponential sums
\begin{align}
\label{eq:intro:ExpSums}
	F_a^{(\bfP)}( \bfalpha ) 
%	= F_a^{(\bfP)}( \bfalpha , \bfP )
	= \sum_{\bfn \in [N]^d} a(\bfn) e( \bfalpha \cdot \bfP( \bfn ) ),
	\quad		
	F^{(\bfP)}( \bfalpha )
	= \sum_{\bfn \in [N]^d} e( \bfalpha \cdot \bfP( \bfn ) )
	\quad
	(\bfalpha \in \T^r).
\end{align}
The circle method expresses the number of solutions to~\eqref{eq:intro:SystPols}
in a subset $A$ of $[N]^d$ as a product of 
$s$ weighted exponential sums of the above form, 
and therefore obtaining bounds on their $s$-th moments
is of major importance.
Restriction theory~\cite{Tao:RestrExpo,Green:RestrCourse,Wolff:HA} 
provides a valuable framework to
derive such bounds.
%of the correct order of magnitude.
When $S$ is a finite subset of $\Z^r$
equipped with a certain measure $\dsigma_S$,
the $L^q \rightarrow L^p$ extension problem is concerned 
with establishing functional estimates of the form
\begin{align*}
	\| (g \dsigma_S)^\wedge \|_{L^p(\T^r)} 
	\leq \| g \|_{\ell^q(S)},
\end{align*}
and it is a dual version of the well-studied restriction problem.
%Motivated by connections to $\Lambda(p)$-sets,
%and to KdV and Schrödinger equations,
Bourgain~\cite{Bourgain:Squares,Bourgain:ParabI,Bourgain:ParabII,Bourgain:SphereI} 
initiated the study of discrete restriction estimates
for the squares, the sphere and the parabola.
Recently, Wooley~\cite{Wooley:Banff,Wooley:Clay} 
has given a formulation of the discrete restriction conjecture
for systems of homogeneous polynomials of dimension one, but the picture is less
clear in higher dimensions.
Short of guessing the right estimates, we put forward a conjecture
which, when it does hold, provides us with exploitable estimates.
We say that $\bfP$ satisfies the discrete
restriction conjecture when it satisfies the estimate
\begin{align}
\label{eq:intro:RestrCrit}
	\| F_a^{(\bfP)} \|_p^p \lesssim_\eps N^{\eps} \| a \|_2^p
\end{align}
in the subcritical range $p < 2K/d$, the $\eps$-full estimate
\begin{align}
\label{eq:intro:RestrEpsFull}
	\| F_a^{(\bfP)} \|_p^p 
	\lesssim_\eps N^{dp/2 - K + \eps} \| a \|_2^p
\end{align}
at the critical exponent $p = 2K/d$, and the $\eps$-free estimate
\begin{align}
\label{eq:intro:RestrEpsFree}
	\| F_a^{(\bfP)} \|_p^p 
	\lesssim_p N^{dp/2 - K} \| a \|_2^p
\end{align}
in the supercritical range $p > 2K/d$.
In the case $d=1$, it is believed that these
estimates all hold~\cite{Wooley:Banff,Wooley:Clay}.
Adding to the existing terminology, 
we say that $\bfP$ satisfies the weak discrete restriction conjecture when
there exists $\theta > 0$ such that
\begin{align}
\label{eq:intro:RestrTrunc}
	\int_{|F_a^{(\bfP)}| \geq N^{d/2-\theta} \|a\|_2} |F_a^{(\bfP)}|^q \ dm 
	\lesssim_q N^{dq/2 - K} \| a \|_2
\end{align}
for $q > 2K/d$.
This weaker estimate is typically easier to obtain, and can be 
used~\cite{Bourgain:Squares,Bourgain:ParabI} to obtain $\eps$-free estimates for exponents $q > p$ 
whenever an $\eps$-full estimate of the form~\eqref{eq:intro:RestrEpsFull} is known.
%This was used by Bourgain~\cite{Bourgain:Squares,Bourgain:ParabI}.

Only supercritical estimates are directly relevant to
our problem, and therefore we quote the literature selectively.
Bourgain established respectively in~\cite{Bourgain:Squares} and~\cite{Bourgain:ParabI} 
that~\eqref{eq:intro:RestrEpsFree} holds 
in the full supercritical range
$p > 4$ for $\bfP = (x^2)$ and $p > 6$ for $\bfP = (x,x^2)$.
%The work of Hu and Li~\cite{HuLi:Degree2,HuLi:Degree3,HuLi:Degreed} 
%handles monomial curves with exponent tuple $\bfk = (1,k)$,
%and in particular it proves~\eqref{eq:intro:MainConj}
%upto $N^\eps$ terms for $p > 14$ when $\bfk = (1,3)$,
%in which case the critical exponent is $2K = 8$.
Keil~\cite{Keil:Diag} found an alternative proof 
of an $L^\infty \rightarrow L^p$ estimate 
%$\eps$-free 
for $p > 6$ when $\bfP = (x,x^2)$.

In the case of the $d$-dimensional parabola
$\bfP = (x_1,\dots,x_d,x_1^2+\dotsb+x_d^2)$, which 
in our terminology is a system of dimension $d$ and weight $d + 2$,
Bourgain~\cite[Propositions~3.82,~3.110,~3.114]{Bourgain:ParabI} 
proved the truncated estimate~\eqref{eq:intro:RestrTrunc}
%for the $d$-dimensional parabola $\bfP = (x_1,\dots,x_d,x_1^2+\dotsb+x_d^2)$
in the whole supercritical range $q > 2(d+2)/d$,
as well as estimates of the form~\eqref{eq:intro:RestrEpsFull}
for $d \in \{ 2,3 \}$, $p > 4$ and for $d \geq 4$, $p \geq 2(d+4)/d$.
Eventually, the powerful decoupling theory of 
Bourgain and Demeter~\cite[Theorem~2.4]{BD:DecouplConj} 
led to the conjectured estimates in all dimensions, that is,~\eqref{eq:intro:RestrCrit} 
and~\eqref{eq:intro:RestrEpsFree} hold respectively for $p = 2(d+2)/d$ and $p > 2(d+2)/d$.
%(using Bourgain's epsilon-removal lemma from~\cite{Bourgain:ParabI}).

There have also been crucial developments 
for systems of polynomials of large degree.
In that setting a natural object is the 
(multidimensional) Vinogradov mean value
\begin{align*}
	J_{s,\bfP}(N) = \int_{\T^r} |F^{(\bfP)}(\bfalpha)|^{2s} \dbfalpha,
\end{align*}
which counts the number of solutions $\bfn_i,\bfm_i \in [N]^d$
to the sytem of equations
\begin{align*}
	\bfP(\bfn_1) + \dotsb + \bfP(\bfn_s) = 
	\bfP(\bfm_1) + \dotsb + \bfP(\bfm_s).
\end{align*}
A bound of the form $J_{\ell,\bfP}(N) \lesssim_\eps N^{2d\ell - K + \eps}$
for an integer $\ell \geq K$ typically allows
for a successful circle method treatment 
of the system of equations~\eqref{eq:intro:SystPols}
in $s > \ell$ variables.

Let us temporarily specialize to the case 
$\bfP = (x,\dots,x^k)$ with $k \geq 2$, where $K = \tfrac{1}{2} k (k+1)$
and $J_{s,\bfP}(N) = J_{s,k}(N)$ is the usual 
Vinogradov mean value~\cite[Chapter~5]{Vaughan:Book}.
We introduce a new definition to facilitate the statement of later results.

\begin{definition}
\label{thm:intro:skDef}
For $k \geq 2$, we let $s_k$ denote the least integer $s \geq K = \tfrac{1}{2}k(k+1)$ such that
$J_{s,k}(N) \lesssim_\eps N^{2s - K + \eps}$ for every $\eps > 0$.
\end{definition}

%Note that the estimate~\eqref{eq:intro:RestrEpsFull} at an exponent $p$
%implies $s_k \leq 2p$.
We restrict to $s_k \geq K$
since a simple averaging argument~\cite[Section~7]{Vaughan:Book} 
shows that $J_{s,k}(N) \gtrsim N^s + N^{2s - K}$.
The Vinogradov mean value conjecture, now a theorem, states that $s_k = K$,
and we discuss briefly the history leading to this result. 
The case $k = 2$ is known to follow from simple divisor considerations.
Classical work of Vinogradov~\cite{Vaughan:Book} 
established an efficient asymptotic bound
$s_k \leq (3 + o_{k \to \infty}(1)) \cdot k^2 \log k$.
In a major achievement,
Wooley~\cite{Wooley:EffI,Wooley:EffII,Wooley:Cubic}
was able to settle the Vinogradov mean value conjecture for $k = 3$
and to obtain the improved bound\footnote{
The stronger bound $s_k \leq k(k-1)$ for $k \geq 4$ was also announced in~\cite{Wooley:Cubic}.
} $s_k \leq k^2 - 1 \sim_{k \to \infty} 2K$ for $k \geq 4$,
using his efficient congruencing method.
In a very recent breakthrough,
Bourgain, Demeter and Guth~\cite{BDG:VinoMeanValue}
have settled the full Vinogradov mean value conjecture,
that is $s_k = K$, in the remaining cases $k \geq 4$,
through a novel method rooted in multilinear harmonic analysis.

Via the circle method~\cite[Section~9]{Wooley:EffI},
it can be shown that $\int_{\T^k} |F^{(x,\dots,x^k)}|^p \lesssim N^{p - K}$ for $p > 2s_k$.
Together with a well-known squaring argument for even 
moments\footnote{
By this we mean the bound $\| F_a^{(\bfP)} \|^{2s}_{2s} \leq \| F^{(\bfP)} \|_s^s \| a \|_2^{2s}$,
which was used for instance by Bourgain~\cite[Proposition~2.36]{Bourgain:ParabI}
and Mockenhaupt and Tao~\cite[Lemma~5.1]{MT:RestrFF}.},
this shows that an $\eps$-free restriction 
estimate of the form~\eqref{eq:intro:RestrEpsFree} 
holds for $p \geq 4 s_k + 2$, and in fact it holds
for $p > 4s_k$ via an observation of Hughes~\cite{Hughes:EpsRemoval}.
Up until the work of Bourgain-Demeter-Guth, 
the best available bounds on Vinogradov mean values would therefore only produce
an asymptotic range $p > (1+o_{k \to \infty}(1)) \cdot 8K$ in such estimates.
Wooley~\cite{Wooley:Restr} was able to essentially halve this range\footnote{
The larger range $p > 2k(k-1)$
%which does not follow from the estimate $s_k = K$ and the squaring argument, 
was also announced in~\cite{Wooley:Clay}.},
showing that~\eqref{eq:intro:RestrEpsFree} holds for
$p > 2k(k+1) \sim_{k \to \infty} 4K$, 
and his method extends to systems of polynomials.
%The Bourgain-Demeter-Guth result together
%with the squaring argument now also provides a range 
%$p > 4K + 2$.. 

We now return to the setting of a general system of polynomials $\bfP$,
and state our main abstract result.
%As mentioned before, we wish
%to work in the greatest possible generality.
%We say that $Z$ is a translation-dilation invariant subset of $(\Q^d)^s$
%when it is stable under translations 
%$(\bfx_j) \mapsto (\bfx_j + \bfu)$, $\bfu \in \Q^d$
%and dilations $(\bfx_j) \mapsto \gamma (\bfx_j)$, $\gamma \in \Q$.
Given a translation-dilation invariant subset $Z$ of $(\Q^d)^s$,
meant to represent a space of trivial solutions
to~\eqref{eq:intro:SystPols}, we define the quantities
\begin{align}
	\label{eq:intro:NDef}
	\calN(N,\bfP,\bflambda) 
	&= \#\{\, \text{solutions $(\bfx_1,\dots,\bfx_s) \in [N]^{ds}$ 
									to~\eqref{eq:intro:SystPols}} \,\}, 
	\\
	\label{eq:intro:NZDef}
	\calN_Z(N,\bfP,\bflambda) 
	&= \#\{\, \text{solutions $(\bfx_1,\dots,\bfx_s) \in [N]^{ds} \cap Z$ 
									to~\eqref{eq:intro:SystPols}} \,\}.
\end{align}
%which counts the number of solutions $\bfx_i, \bfy_i \in [N]^d$
%to $\sum_{i=1}^s F(\bfx_i) = \sum_{i=1}^s F(\bfy_i)$.

\begin{theorem}
\label{thm:intro:SystTslInv}
Let $s \geq 3$ and $\lambda_1,\dots,\lambda_s \in \Z \smallsetminus \{0\}$
be such that $\lambda_1 + \dotsb + \lambda_s = 0$.
Suppose that $\bfP$ is a system of
$r$ homogeneous polynomials of dimension $d$ and weight $K$
such that the system of equations~\eqref{eq:intro:SystPols}
is translation-invariant, and $Z$ is a 
translation-dilation invariant subset of $(\Q^d)^s$.
Suppose that, for a constant $\omega > 0$
depending on $s$ and $\bfP$,
\begin{align}
\label{eq:intro:NumberSolsBounds}
	\calN(N,\bfP,\bflambda) \gtrsim N^{ds - K}
	\quad\text{and}\quad
	\calN_Z(N,\bfP,\bflambda) \lesssim N^{ds - K - \omega}.
\end{align}
Suppose also that there exist 
real numbers $0 < s'' < s' < s$
and $\theta > 0$ 
depending on $s$ and $\bfP$ such that
the following restriction estimates hold:
\begin{align}
	\label{eq:intro:RestrEpsFullLinfty}
	\int_{\T^r} |F_a^{(\bfP)}|^{s''} \ \dm 
	&\lesssim_\eps N^{ds'' - K + \eps} \| a \|_\infty^{s''}, \\
	\label{eq:intro:RestrTruncAgain}
	\int_{ |F_a^{(\bfP)}| \geq N^{d/2 - \theta} \| a \|_2 } | F_a^{(\bfP)} |^{s'} \ \dm
	&\lesssim N^{ds'/2 - K} \| a \|_2^{s'}.
\end{align}
Then there exists a constant $c(\bfP,\bflambda) > 0$
such that, for every subset $A$ of $[N]^d$ of density
at least $2(\log N)^{-c(\bfP,\bflambda)}$, 
there exists a tuple $(\bfx_1,\dots,\bfx_s) \in A^s \smallsetminus Z$
satisfying~\eqref{eq:intro:SystPols}.
\end{theorem}

We first comment on the assumptions of this theorem.
The bounds~\eqref{eq:intro:NumberSolsBounds} essentially mean
that the circle method is successful in estimating the number
of non-trivial solutions to~\eqref{eq:intro:SystPols}.
The restriction estimates~\eqref{eq:intro:RestrEpsFullLinfty}
and~\eqref{eq:intro:RestrTruncAgain} 
are the main analytic information
needed for the argument, and they are stronger
than an $L^\infty \rightarrow L^p$ estimate 
\begin{align*}
%\label{eq:intro:FullRestrLinfty}
	\| F_a^{(\bfP)}\|_p^p \lesssim N^{dp - K} \| a \|_\infty^p
\end{align*}
with $p < s$, used in the method of Roth~\cite{Roth:Roth},
but weaker than an $L^2 \rightarrow L^p$ estimate~\eqref{eq:intro:RestrEpsFree}
%\begin{align}
%\label{eq:intro:FullRestrL2}
%	\| F_a^{(\bfP)}\|_p^p \lesssim N^{dp/2 - K} \| a \|_2^p
%\end{align}
with $p < s$, used in the Heath-Brown-Szemerédi argument~\cite{HB:Roth,Szemeredi:Roth,me:logkeil}.
Note that if we have $J_{\ell}(N,\bfP) \lesssim_\eps N^{2d\ell - K + \eps}$
for an integer $\ell \geq K$, then an $L^\infty \rightarrow L^{2\ell}$ 
estimate of the form~\eqref{eq:intro:RestrEpsFullLinfty}
with $s''=2\ell$ automatically holds\footnote{
This follows from the simple bound 
$\| F_a^{(\bfP)} \|_{2s}^{2s} \leq \| F^{(\bfP)} \|_{2s}^{2s} \| a \|_\infty^{2s}$
for integers $s \geq 1$.}.
For this reason, assumption~\eqref{eq:intro:RestrEpsFullLinfty} is typically 
verified in practice when one is using Vinogradov mean value bounds
to estimate the number of solutions $\calN(N,\bfP,\bflambda)$,
which is the case for systems of large degree.

Theorem~\ref{thm:intro:SystTslInv} constitutes an abstract generalization
of its predecessor~\cite[Theorem~2]{me:logkeil},
and its proof is very similar in dimension one 
when a full $L^2 \rightarrow L^p$ restriction estimate 
of the form~\eqref{eq:intro:RestrEpsFree} is known.
In the extension to the multidimensional setting,
the only substantial change to the original energy increment strategy
occurs in the technical linearization part of the 
argument~\cite[Section~9]{me:logkeil}, 
and there we employ the framework of
factors introduced to additive combinatorics 
by Green and Tao~\cite{Green:Factors,Tao:Factors} 
to handle effectively the computations in higher dimensions.
Finally, we need a new observation to exploit truncated restriction estimates
of the form~\eqref{eq:intro:RestrTruncAgain} instead of complete ones, which is that
for the kind of weight functions that arise in the energy increment iteration,
one can afford to ignore the moment tails of associated exponential sums.

We now discuss several consequences of Theorem~\ref{thm:intro:SystTslInv},
starting with the one-dimensional setting.
There the only translation-invariant system of equations
of the form~\eqref{eq:intro:SystPols} up to equivalence is
\begin{align}
	\label{eq:intro:SystMonom}
	&\phantom{(1 \leq j \leq k)} &
	\lambda_1 x_1^j + \dotsb + \lambda_s x_s^j &= 0
	&&(1 \leq j \leq k),
\end{align}
corresponding to $\bfP = (x,\dots,x^k)$.
Using the optimal bound $s_k = 2K$ 
to verify the assumptions~\eqref{eq:intro:NumberSolsBounds}
and~\eqref{eq:intro:RestrEpsFullLinfty}
of Theorem~\ref{thm:intro:SystTslInv},
as well as a certain truncated restriction estimate of our own,
we obtain the following conclusion.

\begin{theorem}[Additive equations in subsets of monomial curves]
\label{thm:intro:SystMonom}
Let $k \geq 3$ and $K = \tfrac{1}{2}k(k+1)$. 
Let $s \geq 3$ and $\lambda_1, \dots, \lambda_s \in \Z \smallsetminus \{0\}$
be such that $\lambda_1 + \dotsb + \lambda_s = 0$.
Suppose that the system of equations~\eqref{eq:intro:SystMonom}
possesses nonsingular real and $p$-adic solutions for every prime $p$.
When $s > 2K + 4$,
there exists a constant $c(k,\bflambda) > 0$ 
such that every subset of $[N]$ of density 
at least $2(\log N)^{-c(k,\bflambda)}$
contains a solution to the system of equations~\eqref{eq:intro:SystMonom},
which is neither a projected nor a subset-sum solution.
\end{theorem}

Note that, critically, our approach bypasses the need
for complete $L^2 \rightarrow L^p$ restriction estimates,
which are at present only known~\cite{Wooley:Restr} 
for $p > 2k(k+1) \sim_{k \to \infty} 4K$.
For this reason, we are able to reach a number $s$ of variables close
to the limit of the circle method, which is $s > 2K$ in this setting.
Furthermore, this number of variables could be attained
if one only knew the truncated estimate~\eqref{eq:intro:RestrTrunc}
in the range $p > 2K$.

For general systems of polynomials of large degree,
the most general conclusion we can obtain is the following,
of which Theorem~\ref{thm:intro:SystMonomMultidim}
is a special case.

\begin{theorem}[Additive equations in subsets of polynomial surfaces]
\label{thm:intro:SystPols}
Let $s \geq 1$ and $\lambda_1, \dots, \lambda_s \in \Z \smallsetminus \{0\}$
be such that $\lambda_1 + \dotsb + \lambda_s = 0$.
Suppose that $\bfP$ is a reduced translation-dilation invariant
system of polynomials having 
dimension $d$, rank $r$, degree $k$ and weight $K$.
Suppose also that the system of equations~\eqref{eq:intro:SystPols}
possesses nonsingular real and $p$-adic solutions for every prime $p$.
When $k \geq 2$ and $s > \max( 2r(k+1) , K^2 + d )$, 
there exists a constant $c(\bfP,\bflambda) > 0$
such that every subset of $[N]^d$ of density at least 
$2(\log N)^{-c(\bfP,\bflambda)}$
contains a solution to the system of equations~\eqref{eq:intro:SystPols},
which is neither a projected nor a subset-sum solution.
\end{theorem}

To prove this result, one may choose to appeal to
either the $L^2 \rightarrow L^p$ restriction estimates of Wooley~\cite{Wooley:Restr},
or to weaker truncated restriction estimates that we will provide.
The assumptions~\eqref{thm:intro:SystTslInv} on the number
of integer solutions are verified by quoting
the asymptotic formulas of Parsell et al.~\cite{PPW:Multidim},
based on the efficient congruencing method.
As a parenthesis, we remark that in the special case
where the coefficients $(\lambda_i)$ in~\eqref{eq:intro:SystPols}
take a symmetric form $(\mu_1,-\mu_1,\dots,\mu_\ell,-\mu_\ell)$,
a simple Cauchy-Schwarz argument
yields the conclusion of Theorems~\ref{thm:intro:SystMonomMultidim},~\ref{thm:intro:SystMonom}
and~\ref{thm:intro:SystPols} at power-like densities $N^{-c(\bfP)}$ instead
%for only $2K + 2$ variables.
(see Proposition~\ref{thm:large:SystPolsSym} below).
It is expected~\cite{Bourgain:VinoSurvey} that the 
decoupling theory of Bourgain-Demeter-Guth could also lead to
to progress on bounds for multidimensional Vinogradov mean values,
which could in turn improve the range of validity of Theorem~\ref{thm:intro:SystPols}.

%The additional constraint $s > K^2 + d$ serves to 
%avoid the scenario where the solution we find belongs
%to a certain subvariety, as explained by Parsell et al.~\cite{PPW:Multidim}.
Finally, we consider the parabola system
\begin{align}
\label{eq:intro:SystParab}
	\begin{split}
	\lambda_1 \bfx_1 + \dotsb + \lambda_s \bfx_s &= 0, \\
	\lambda_1 |\bfx_1|^2 + \dotsb + \lambda_s |\bfx_s|^2 &= 0
	\end{split}
\end{align}
in variables $\bfx_1,\dots,\bfx_s \in \Z^d$,
which corresponds to the system of polynomials
\begin{align*}
	\bfP = (x_1,\dots,x_d,x_1^2 + \dotsb + x_d^2),
\end{align*}
generated by the seed polynomial $P(\bfx) = |\bfx|^2$.
When all the $\lambda_i$ but one have the same sign,
say all but $\lambda_s$, every solution $\bfx$ to~\eqref{eq:intro:SystParab}
verifies 
\begin{align*}
	\lambda_1 |\bfx_1-\bfx_s|^2 + \dotsb + \lambda_{s-1} |\bfx_{s-1} - \bfx_s|^2 = 0
\end{align*}
by translation-invariance,
and by definiteness we have $\bfx_1 = \dots = \bfx_s$.
Barring this unfortunate circumstance, which always occurs for $s = 3$, 
we can obtain a positive result for a number of dense variables
exceeding the critical exponent $p_d = 2(d+2)/2$ of the discrete parabola,
which directly generalizes~\cite[Theorem~2]{me:logkeil}.
%We consider trivial the solutions $(\bfx_1,\dots,\bfx_s)$ 
%to the system~\eqref{eq:intro:SystParab}
%such that $\bfx_i = \bfx_j$ for distinct indices $i,j \in [s]$,
%as well as subset-sum solutions.

\begin{theorem}[Additive equations in subsets of the parabola]
\label{thm:intro:SystParab}
Let $d,s \geq 1$ and suppose that
$\lambda_1, \dots, \lambda_s \in \Z \smallsetminus \{0\}$
are such that $\lambda_1 + \dotsb + \lambda_s = 0$
and at least two of the $\lambda_i$ are positive,
and at least two are negative.
There exists a constant $c(d,\bflambda) > 0$ 
such that every subset of $[N]^d$ of density 
at least $2(\log N)^{-c(d,\bflambda)}$
contains a solution 
to the system of equations~\eqref{eq:intro:SystParab},
which is neither a subset-sum solution
nor a solution with two equal coordinates,
provided that
\begin{enumerate}
	\item $d = 1$ and $s \geq 7$, or
	\item $d = 2$ and $s \geq 5$, or
	\item $d \geq 3$ and $s \geq 4$.
\end{enumerate}
\end{theorem}

This result takes as input the aforementioned Strichartz estimates
of Bourgain and Demeter~\cite{BD:DecouplConj}
%(Theorem~\ref{thm:intro:RestrParab})
to verify the assumptions~\eqref{eq:intro:RestrEpsFullLinfty} 
and~\eqref{eq:intro:RestrTruncAgain}
of Theorem~\ref{thm:intro:SystTslInv}, 
while a lower bound for the number of solutions to~\eqref{eq:intro:SystParab}
can be obtained by reducing the system to a quadratic form
of rank at least five.
For dimensions $d \not\in \{ 3,4 \}$,
or for $d \in \{3,4\}$ and $s \geq 5$ variables,
earlier estimates of Bourgain~\cite{Bourgain:ParabI} are in fact
sufficient for our analysis.

Another use of restriction estimates
for the parabola that we wish to highlight
is to obtain an asymptotic formula for
the number of solutions to~\eqref{eq:intro:SystParab}
in a box $[N]^d$, under local solvability assumptions.

\begin{theorem}
\label{thm:intro:SystParabAsympt}
Let $d,s \geq 1$ and $\lambda_1, \dots, \lambda_s \in \Z \smallsetminus \{0\}$.
Suppose that the system of equations~\eqref{eq:intro:SystParab}
has a nonsingular real solution in $(0,+\infty)^{ds}$
and nonsingular $p$-adic solutions for every prime $p$.
Let $\calN(N,d,\bflambda)$ denote the number of solutions 
to~\eqref{eq:intro:SystParab} in $[N]^d$.
For $s > 2 + \frac{4}{d}$, we have
\begin{align*}
	\calN(N,d,\bflambda) \sim \frakS \cdot \frakJ \cdot N^{ds-(d+2) }
\end{align*}
as $N \rightarrow \infty$, where $\frakS$, $\frakJ > 0$.
\end{theorem}

The factors $\frakS$ and $\frakJ$ are defined in~\eqref{eq:parab:SgSeries}
and~\eqref{eq:parab:SgIntg} below (with $T=\infty$),
and through further analysis they be given the
traditional interpretation in terms
%~\cite[Chapter~20]{IK:Book} in terms 
of products of local densities associated
to the system of equations~\eqref{eq:intro:SystParab},
though we do not provide the details here.
When counting solutions to~\eqref{eq:intro:SystParab} in $[-N,N]^d \cap \Z^d$
instead, one needs only assume the existence of a nonzero real solution
to~\eqref{eq:intro:SystParab}, as we explain in Section~\ref{sec:parab}.
The approach by reduction to a quadratic form
is also likely to produce an asymptotic formula,
but it is not clear that one would recover the 
same expression for local densities.

% RK Mention Behrend ?
%Perhaps importantly, it should be noted that
%the systems of equations of Theorems~\ref{thm:intro:SystMonomMultidim}
%to~\ref{thm:intro:SystParab} all contain a linear equation,

%%% CLOSING

We close this already lengthy introduction by 
discussing certain limitations of the previous results.
First, an annoying feature of Theorem~\ref{thm:intro:SystTslInv} is the dependency
of the logarithm exponent on the coefficients $(\lambda_i)$
and the system of polynomials $\bfP$.
This is a seemingly irreducible feature of the 
Heath-Brown-Szemerédi argument~\cite{HB:Roth,Szemeredi:Roth} 
which is not present in other methods such as Roth's~\cite{Roth:Roth}.
Secondly, our approach does not yield 
the expected density of solutions $c(\delta) N^{ds - K}$
to the equations~\eqref{eq:intro:SystPols} 
in a subset of density $\delta$ of a box $[N]^d$,
and it would be very desirable to find a density increment
strategy that addresses this shortcoming\footnote{
This question was raised to the author by \'{A}kos Magyar, whom we thank here.}.
For systems given by one quadratic form which is in a sense
far from being diagonal (that is, with large off-rank),
Keil~\cite{Keil:QuadFormsI,Keil:QuadFormsII} has devised such a strategy,
which relies on finding a uniform majorant of weighted exponential sums 
by Weyl differencing.
However, it seems difficult to obtain such bounds in the diagonal situation,
where the weights are not easily eliminated,
and we anticipate that a set of techniques 
involving Bohr sets might be required instead.

\textbf{Remark.}
A prior version of this article was publicized before the
announcement of Bourgain, Demeter and Guth~\cite{BDG:VinoMeanValue}.
This new version records the consequences of this new development
for some of our estimates.

\textbf{Acknowledgements.}
We thank Lilian Matthiesen for an interesting remark
which inspired Proposition~\ref{thm:large:SystPolsSym}.
We thank Trevor Wooley for communicating us
an advanced copy of his forthcoming manuscript~\cite{Wooley:Restr}.
This work was supported by NSERC Discorery grants 22R80520
and 22R82900.

\section{Notation}
\label{sec:notat}

For $x \in \R$ and $q \in \N$,
we write $e(x) = e^{2i\pi x}$ and $e_q(x)=e(\frac{x}{q})$.
For functions $f : \T^d \rightarrow \C$
and $g : \Z^d \rightarrow \C$, we define
$\wh{f}( \bfk ) = \int_{\T^d} f( \bfalpha ) e( - \bfk \cdot \bfalpha ) \dbfalpha$
and $\wh{g}( \bfalpha ) = \sum_{\bfn \in \Z^d} g(\bfn) e( \bfalpha \cdot \bfn ) $.
%For a function $h : \R^d \rightarrow \C$
%we let $\wh{h}(\bfxi) = \int_{\R^d} f(\bfx) e( - \bfxi \cdot \bfx ) \dbfx$.
For a function $f$ defined on abelian group $G$
and $x,t \in G$, we let $\tau_t f(x) = f(x+t)$.

When $k \geq 1$, $\bfa \in \Z^k$ and $q \in \N$, we write
$(\bfa,q) = \gcd (a_1,\dots,a_k,q)$, and we let 
$q | \bfa$ denote the fact that $q | a_1,\dots, q | a_k$.
For $q \geq 2$ we occasionally use $\Z_q$ as a shorthand for the group $\Z/q\Z$.
We write $\| x \|$ or sometimes $\| x \|_\T$ 
for the distance of a real $x$ to $\Z$.

We let $\dm$ denote the Lebesgue measure on $\R^d$, or on 
$\T^d$ identified with any cube of the form $[-\theta,1-\theta)^d$,
and we let $\dSigma$ denote the counting measure on $\Z^d$.

When $\Omega$ is a finite set and $f : \Omega \rightarrow \C$ is a function,
we write $\E_\Omega f = \E_{x \in \Omega} f(x) = |\Omega|^{-1} \sum_{x \in \Omega} f(x)$.
When $\calP$ is a property, we let $1_{\calP}$ or $1[ \calP ]$ denote the boolean
which equals $1$ when $\calP$ is true, and $0$ otherwise.
When $n$ is an integer we write $[n] = \{1,\dots,n\}$,
and we let $\N_0 = \N \cup \{0\}$.
We let $A \bigsqcup B$ denote the disjoint union of sets $A$ and $B$.

\section{Additive equations in dense variables}
\label{sec:add}

In this section, we prove Theorem~\ref{thm:intro:SystTslInv}.
We employ the arithmetic energy-increment method from our previous work~\cite{me:logkeil},
with several simplifications to make the high-dimensional framework more bearable, 
and with a more significant modification to use truncated restriction estimates.

We start by introducing the relevant objects.
We fix a system of $r$ homogeneous polynomials 
$\bfP = (P_1,\dots,P_r)$, where each $P_i \in \Z[x_1,\dots,x_d]$
has degree $k_i \geq 1$, and we recall that
$k = \max_{1 \leq i \leq r} k_i$ is the degree of $\bfP$
and $K = k_1 + \dotsb + k_r$ is its weight.
We also fix coefficients $\lambda_1,\dots,\lambda_s \in \Z \smallsetminus \{0\}$
such that $\lambda_1 + \dotsb + \lambda_s = 0$.
We fix an integer $N \geq 2$ and we study the system of equations
\begin{align}
\label{eq:add:SystPols}
	\lambda_1 \bfP(\bfn_1) + \dotsb + \lambda_s \bfP(\bfn_s) = 0
\end{align}
in variables $\bfn_1,\dots,\bfn_d \in [N]^d$.
We also fix a translation-dilation invariant subset $Z$ of $(\Q^d)^s$,
to be thought of as a set of trivial solutions to~\eqref{eq:add:SystPols},
and we define the quantities $\calN(N,\bfP,\bflambda)$ and $\calN_Z(N,\bfP,\bflambda)$ as 
in~\eqref{eq:intro:NDef} and~\eqref{eq:intro:NZDef}.
From now on, we place ourselves under the assumptions of Theorem~\ref{thm:intro:SystTslInv},
which in particular imply that $N$ can be taken larger than any fixed
constant depending on $\bfP$ and $\bflambda$.
Unless otherwise specified, all explicit and implicit constants 
throughout the section may depend on $\bfP$ and $\bflambda$.
% OPT
% Say C_{i+1} may depend on C_i in single statements?

Next, we fix a prime number $M \sim D N$, where $D = D(\bfP,\bflambda) > 0$
is chosen large enough so that $(M,\lambda_i) = 1$ for all $i$ and
so that, for $\bfn_1,\dots,\bfn_s \in [N]^d$, 
the system of equations~\eqref{eq:add:SystPols} is equivalent to 
\begin{align}
\label{eq:add:SystPolsModulo}
	&\phantom{(1 \leq k_j \leq r)}	&
	\lambda_1 P_j(\bfn_1) + \dotsb + \lambda_s P_j(\bfn_s) &\equiv 0 \bmod M^{k_j}
	&&(1 \leq j \leq r).
\end{align}
Accordingly we define $\ZM = \prod_{j=1}^r \Z/M^{k_j}\Z$ ; note that $|\ZM| = M^K \asymp N^K$.
When $f : \Z^d \rightarrow \C$ is a function, 
we also define $F_f : \T^r \rightarrow \C$ 
and $H_f : \Z_{\overline{\bfM}} \rightarrow \C$ by
\begin{align}
\label{eq:add:FfHfDef}
	F_f( \bfalpha ) = \sum_{\bfn \in [N]^d}\,f(\bfn) e\bigg( \sum_{j=1}^r \alpha_j P_j(\bfn) \bigg),
	\qquad
	H_f( \bfxi ) = \E_{\bfn \in [N]^d}\,f(\bfn) e\bigg( \sum_{j=1}^r \frac{\xi_j P_j(\bfn)}{M^{k_j}} \bigg),
\end{align}
so that $H_f(\bfxi) = N^{-d} F_f(\xi_1/M^{k_1},\dots,\xi_r/M^{k_r})$
and $F_f = F_f^{(\bfP)}$ in the notation of the introduction.
We write respectively $F$ and $H$ for the 
unweighted versions of $F_f$ and $H_f$ where one takes $f \equiv 1$.
For $p > 0$, we define the $\ell^p$ norm of a function $G : \ZM \rightarrow \C$ by
$\| G \|_p = ( \sum_{\bfxi \in \ZM} |G(\bfxi)|^p )^{1/p}$.

Next, we define the multilinear operator
$T$ acting on functions $f_i : \Z^d \rightarrow \C$ by
\begin{align}
\label{eq:add:TDef}
	T(f_1,\dots,f_s) 
	= \frac{D^K}{N^{ds - K}}
	\sum_{\bfn_1,\dots,\bfn_s \in [N]^d} f_1(\bfn_1) \cdots	f_s(\bfn_s) 
	1\bigg[\, \sum_{i=1}^s \lambda_i \bfP(\bfn_i) = 0 \,\bigg].
\end{align} 
The normalizing constant $D$ is unimportant
and will be eventually absorbed in big $O$ notation.
Note that $T(1_{[N]^d},\dots,1_{[N]^d}) = D^K N^{-(ds - K)} \calN(N,\bfP,\bflambda)$.
As mentioned in the introduction, a fact of key importance to us
is that the operator $T$ is controlled by $s$-th moments of 
the exponential sums $H_f$.

\begin{proposition}
For functions $f_1,\dots,f_s : \Z^d \rightarrow \C$, we have
\begin{align}
%	\notag
%	T(f_1,\dots,f_s)
%	&= \sum_{\bfxi \in \ZM}
%	H_{f_1}( \lambda_1 \bfxi ) \dotsb H_{f_s}( \lambda_s \bfxi ) \\
	\label{eq:add:HolderOnT}
%	\text{and}\qquad
	|T(f_1,\dots,f_s)| \leq \| H_{f_1} \|_s \dotsb \| H_{f_s} \|_s.
\end{align}
\end{proposition}

\begin{proof}
For convenience we define the bilinear form
$\langle \bfx , \bfy \rangle = \sum_{j=1}^r x_j y_j M^{-k_j}$ on $\ZM$.
By equivalence of~\eqref{eq:add:SystPols} and~\eqref{eq:add:SystPolsModulo}
for $\bfn_i \in [N]^d$ and by orthogonality, we have
\begin{align*}
	T(f_1,\dots,f_s)
	&= \frac{D^K}{N^{ds - K}}
	\sum_{\bfn_1,\dots,\bfn_d \in [N]^d} f_1(\bfn_1) \cdots	f_s(\bfn_s)
	\frac{1}{M^K} \sum_{\bfxi \in \ZM} e( \langle \bfxi , \lambda_1 \bfP( \bfn_1 ) + \dotsb + \lambda_s \bfP( \bfn_s ) \rangle )
\end{align*}
Interchanging summations, and renormalizing, we obtain
\begin{align*}
	T(f_1,\dots,f_s)
	&= \sum_{\bfxi \in \ZM} \E_{\bfn_1,\dots,\bfn_s \in [N]^d} 
	f_1(\bfn_1) e( \langle \lambda_1 \bfxi , \bfP( \bfn_1 ) \rangle )
	\cdots f_s( \bfn_s ) e(  \langle \lambda_s \bfxi , \bfP( \bfn_s ) \rangle ) \\
	&= \sum_{\bfxi \in \ZM}
	H_{f_1}( \lambda_1 \bfxi ) \dotsb H_{f_s}( \lambda_s \bfxi ).
\end{align*}
By Hölder's inequality, we deduce that
\begin{align*}
	|T( f_1 , \dots, f_s )| \leq \prod_{i=1}^s \| H_{f_i}( \lambda_i \,\cdot ) \|_s.
\end{align*}
For every $i \in [s]$, we have $\| H_{f_i}( \lambda_i \,\cdot ) \|_s = \| H_{f_i} \|_s$,
since the $M^{k_j}$, $j \in [r]$ are all coprime to $\lambda_i$,
and this concludes the proof.
\end{proof}

The exponential sums $H_f$, being discretized versions of $F_f$,
behave exactly the same insofar as moments are concerned.

\begin{lemma}
\label{thm:add:HBoundByF}
Uniformly for functions 
$f : [N]^d \rightarrow \C$, 
we have, for every $p \geq 1$,
\begin{align*}
	\| H_f \|_p^p
	\lesssim_p N^{K-dp} \| F_f \|_p^p.
\end{align*}
\end{lemma}

\begin{proof}
Define $g : \Z^r \rightarrow \C$ by
\begin{align*}
	g( \bfm ) = \sum_{ \bfn \in [N]^d \,:\, \bfP(\bfn) = \bfm } f( \bfn ),
\end{align*}
so that $F_f = \wh{g}$ by~\eqref{eq:add:FfHfDef}.
By~\cite[Proposition~6.1]{me:logkeil}, we have therefore
\begin{align*}
	\| H_f \|_p^p
	&= N^{-dp} \sum_{\xi_1 \in \Z/M^{k_1}\Z}
	\dots \sum_{\xi_r \in \Z/M^{k_r}\Z}
	\Big| \wh{g} \Big( \frac{\xi_1}{M^{k_1}} , \dots , \frac{\xi_r}{M^{k_r}} \Big) \Big|^p
	\\
	&\lesssim_p N^{K-dp} \int_{\T^r}
	| \wh{g}( \theta_1 , \dots , \theta_r ) |^p 
	\dtheta_1 \dots \dtheta_r
	\\
	&= N^{K-dp} \| F_f \|_p^p. \phantom{\int_\T}
\end{align*}
\end{proof}

We also need a technical lemma
to transform the assumptions of Theorem~\ref{thm:intro:SystTslInv}
into useful restriction estimates.
It is more natural at this point to
work with scaled averages, 
and thus for a function $f : [N]^d \rightarrow \C$ and $p > 0$
we define  $\| f \|_{L^p[N]} = (\E_{\bfn \in [N]^d} |f(\bfn)|^p)^{1/p}$.

\begin{lemma}
\label{thm:csq:RestAlmostCstFct}
Let $d,r \geq 1$, $\theta > 0$ and $0 < q < p$.
Suppose that $T : \ell^1(\Z^d) \rightarrow L^\infty(\T^r)$ is an operator 
such that, for every $\eps > 0$,
\begin{align}
	\label{eq:csq:RestrInfty}
	\int_{\T^r} |T f|^q \ \dm 
	\lesssim_\eps N^{dq - K + \eps} \| f \|_\infty^q, \\
	\label{eq:csq:RestrL2Trunc}
	\int_{|T f| \geq N^{d - \theta} \| f \|_{L^2[N]} } |T f|^p \ \dm
	\lesssim N^{dp - K} \| f \|_{L^2[N]}^p.
\end{align}
Then, uniformly for functions $f : [N]^d \rightarrow \C$, we have
\begin{align*}
	\| T f \|_p^p 
	\,\lesssim_{p,q,\theta} N^{dp - K} \| f \|_{L^2[N]}^{p-q} \| f \|_\infty^q.
\end{align*}
Furthermore, for $0 < \nu < (\tfrac{p}{q}-1) \theta$ we have,
uniformly for functions $f : [N]^d \rightarrow \C$
such that $\| f \|_\infty / \| f \|_{L^2[N]} \leq N^\nu$,
\begin{align*}
	\| T f \|_p^p 
	\,\lesssim_{p,q,\theta,\nu} N^{dp - K} \| f \|_{L^2[N]}^p.
\end{align*}
\end{lemma}

\begin{proof}
Since $\| f \|_{L^2[N]} \leq \| f \|_\infty$
and we have the estimate~\eqref{eq:csq:RestrL2Trunc},
it suffices in both cases to bound the tail
\begin{align*}
	I = \int_{|T f| \leq N^{d - \theta} \| f \|_{L^2[N]} } |T f|^p  \ \dm.
\end{align*}
To obtain the first estimate, observe that by~\eqref{eq:csq:RestrInfty} we have
\begin{align*}
	I 
	&\leq N^{(p-q)(d-\theta)} \| f \|_{L^2[N]}^{p-q} \int_{\T^r} |T f|^q \ \dm
	\\
	&\lesssim_\eps N^{\eps - (p-q)\theta} N^{dp - K} \| f \|_{L^2[N]}^{p-q} \| f \|_\infty^q.
\end{align*}
For $\eps$ small enough, 
%and using nesting of $L^p[N]$ norms, 
we obtain the first estimate.
To obtain the second estimate, 
note that when $\| f \|_\infty \leq N^\nu \| f \|_{L^2[N]}$, we have
\begin{align*}
	I \leq N^{\eps + q\nu - (p-q)\theta} N^{p - K} \| f \|_{L^2[N]}^p.
\end{align*}
For $\nu < (\tfrac{p}{q} - 1) \theta$ and $\eps$ small enough, 
we obtain the second estimate.
\end{proof}

Using the previous lemmas, 
we can translate these assumptions
into a simple $L^2 \rightarrow L^p$
estimate for the operator $f \mapsto H_f$ 
acting on functions of small $L^\infty$/$L^2$ ratio,
and into an inhomogeneous 
``mixed norms'' estimate for general functions.

\begin{proposition}
\label{thm:add:HBounds}
Uniformly for functions $f : [N]^d \rightarrow \C$, we have
\begin{align}
\label{eq:add:HBoundLinfty}
	\| H_f \|_p \lesssim \| f \|_{L^2[N]}^{1-(s''/s')} \| f \|_\infty^{s''/s'} \leq \| f \|_\infty
	\quad\text{for}\quad p \geq s'.
\end{align}
There exists a constant $\nu \in (0,1]$ 
depending at most on $s',s'',\theta$ such that,
uniformly for functions $f : [N]^d \rightarrow \C$
such that ${\| f \|_\infty \leq 1}$ and 
$\| f \|_{L^2[N]} \geq N^{-\nu}$, we have
\begin{align}
\label{eq:add:HBoundL2}
	\| H_f \|_p &\lesssim \| f \|_{L^2[N]}
	\quad\text{for}\quad p \geq s'.
\end{align}
\end{proposition}

\begin{proof}
By reverse nesting of $\ell^p(\ZM)$ norms, it
suffices to prove both estimates 
at the endpoint $s'$.
We rewrite the assumptions~\eqref{eq:intro:RestrEpsFullLinfty} 
and~\eqref{eq:intro:RestrTruncAgain} as
\begin{align*}
	\int_{\T^r} |F_f|^{s''} \dm 
	&\lesssim_\eps N^{ds'' - K + \eps} \| f \|_\infty^{s''}, \\ 
	\int_{ |F_f| \geq N^{d - \theta} \| f \|_{L^2[N]} } |F_f|^{s'} \dm 
	&\lesssim N^{ds' - K} \| f \|_{L^2[N]}^{s'},
\end{align*}
where $0 < s'' < s' < s$ and $\theta > 0$.
The proof follows by applying Lemma~\ref{thm:csq:RestAlmostCstFct} 
to $Tf = F_f$ with $(q,p)=(s'',s')$ and $\nu = \tfrac{1}{2}(\tfrac{s''}{s'} - 1) \theta$,
and then invoking Lemma~\ref{thm:add:HBoundByF}.
%to obtain
%\begin{align*}
%	\| F_f \|_{s'}^{s'} &\lesssim N^{dp - K} \| f \|_{L^2[N]}^p
%	&&(f : N \rightarrow \C),
%	\\
%	\| F_f \|_{s'}^{s'} &\lesssim N^{dp - K} \| f \|_{L^2[N]}^p.
%	&&(f : N \rightarrow \C \ \text{such that}\ \|f\|_{\infty}/\|f\|_{L^2[N]} \leq N^\nu).	
%\end{align*}
%The proposition then follows from~\eqref{eq:add:HBoundByF}.
\end{proof}

With the previous analytical tools in place,
we can carry out the first step of the usual density increment strategy,
which is to extract a large moment of the exponential sum $H_f$.
When $A$ is a subset of $[N]^d$ of density $\delta$, 
we write $f_A = 1_A - \delta 1_{[N]^d}$ for its balanced
indicator function, here and throughout the section.

\begin{proposition}
\label{thm:add:LargeMoment}
There exists a constant $c_0 > 0$ such that
the following holds.
If $A$ is a subset of $[N]^d$ of density $\delta$
such that $T(1_A,\dots,1_A) \leq c_0\, \delta^s$, then
\begin{align*}
	1 \lesssim \| H_{f_A / \delta} \|_s.
\end{align*}
\end{proposition}

\begin{proof}
We expand $1_A = f_A + \delta 1_{[N]^d}$ by multilinearity in
\begin{align*}
	O( c_0 \delta^s )
	&= T( 1_A, \dots, 1_A ) \\
	&= \delta^s T( 1_{[N]^d}, \dots, 1_{[N]^d} )
	+ \textstyle \sum T( \ast,\dots,f_A,\dots,\ast ) \\
	&= \delta^s D^K N^{-(ds-K)} \calN(N,\bfP,\bflambda)
	+ \textstyle \sum T( \ast,\dots,f_A,\dots,\ast ),
\end{align*}
where the sum is over $2^s - 1$ terms and the asterisks
denote functions equal to $f_A$ or $\delta 1_{[N]^d}$.
Recalling the assumption~\eqref{eq:intro:NumberSolsBounds},
we assume that $c_0$ is small enough and use the pigeonhole principle 
to obtain a lower bound of the form
\begin{align*}
	\delta^s \lesssim |T(f_1,\dots,f_s)|,
\end{align*}
where a number $\ell \geq 1$ of the functions $f_i$ are equal to $f_A$,
and others are equal to $\delta 1_{[N]^d}$.
Therefore, by~\eqref{eq:add:HolderOnT} and~\eqref{eq:add:HBoundLinfty},
we have
\begin{align*}
	\delta^s
	\lesssim \| H_{f_A} \|_s^\ell \cdot \delta^{s-\ell} \| H \|_s^{s-\ell}
	\lesssim \delta^{s-\ell} \| H_{f_A} \|_s^\ell.
\end{align*}
After some rearranging we find that
$\delta \lesssim \| H_{f_A} \|_s$,
which finishes the proof.
\end{proof}

The next step is identical to that in the one-dimensional case~\cite[Section~8]{me:logkeil}:
we extract a large restricted moment involving few frequencies.

\begin{proposition}
\label{thm:add:LargeRestrictedMoment}
There exist positive constants $c_0$, $c_1$, $C_1$ such that
the following holds.
If $A$ is a subset of $[N]^d$ of density $\delta$
such that $T(1_A,\dots,1_A) \leq c_0 \delta^s$, then
there exists $1 \leq R \leq (\delta/2)^{-C_1}$
and distinct frequencies $\bfxi_1,\dots,\bfxi_R \in \ZM$ such that
\begin{align*}
	R^{c_1} \lesssim \sum_{i=1}^R |H_{f_A/\delta}(\bfxi_i)|^{s'}.
\end{align*}
\end{proposition}

\begin{proof}
By Proposition~\ref{thm:add:LargeMoment} 
and~\eqref{eq:add:HBoundLinfty}, we have
\begin{align*}
	1 \lesssim \sum_{\bfxi} |H_{f_A/\delta}(\bfxi)|^s,
	\qquad
	\sum_{\bfxi} |H_{f_A/\delta}(\bfxi)|^{s'} \lesssim \delta^{-s'}.
\end{align*}
The proposition then follows at once from~\cite[Lemma~8.1]{me:logkeil}
upon reordering the $|H_{f_A/\delta}(\bfxi)|$ by size.
\end{proof}

The next stage of the arithmetic Heath-Brown-Szemerédi method
requires an estimate of simultaneous diophantine approximation
essentially due to Schmidt~\cite[Chapter~7]{Baker:Book}
and refined by Green and Tao~\cite[Proposition~A.2]{GT:DiophApprox}.
Here we use the more general version of Lyall and Magyar~\cite[Proposition~B.2]{LM:DiophApprox},
which applies to monomials of arbitrary degree.

\begin{proposition}
\label{thm:add:SimultDiophApprox}
Let $L,T \in \N$ and $k \geq 1$.
There exist constants $c, C > 0$ 
depending at most on $k$
such that, for any
$\theta_1,\dots,\theta_T \in \R$
and for $L \geq (2T)^{C T^2}$,
there exist $1 \leq q \leq L$ such that
%\footnote{When $R_j = 0$, we consider that condition void.}
$\| q^k \theta_i \|_\T \leq L^{- c T^{-2}}$
for all $1 \leq i \leq T$.
\end{proposition}

We define a cube progression as
a set of the form $\bfu + q [L]^d$
with $\bfu \in \Z^d$ and $q,L \geq 1$.
We define a polynomial phase function $\phi : \Z^d \rightarrow \T$
simply as a map $\phi(\bfx) = G(\bfx) \bmod 1$,
for a polynomial $G \in \R[x_1,\dots,x_d]$, and we define\footnote{
This is a slight abuse of notation, since $G$ is not uniquely defined from $\phi$,
but in practice we consider polynomial phase functions as formal couples $(\phi,G)$.}
the degree of $\phi$ to be that of $G$.
When $Q$ is a subset of $\Z^d$ and $\phi : \Z^d \rightarrow \T$ 
is a polynomial phase function, we let
\begin{align*}
	\diam_Q(\phi) = \sup_{\bfx,\bfy \in Q} \| \phi(\bfx) - \phi(\bfy) \|_\T.
\end{align*}
With this vocabulary in place, we now carry out a familiar linearization procedure.

\begin{proposition}[Simultaneous linearization of polynomial phases]
\label{thm:add:LinPolPhases}
Let $k \geq 0$ and $d \geq 1$.
There exist constants $c, C > 0$
depending at most on $k$ and $d$ such that
the following holds.
Let $R \geq 1$ and suppose that 
%$E$ is a finite set of polynomial phase functions
$\phi_1,\dots,\phi_R : \Z^d \rightarrow \T$ 
are polynomial phase functions
such that $\phi_j(0) = 0$ and $\deg \phi_j \leq k$
for all $j \in [R]$.
Assume that $N \geq (2R)^{C R^{2k}}$.
Then there exists a partition of the form 
${ [N]^d = ( \bigsqcup_i Q_i ) \bigsqcup \Xi }$,
where each $Q_i$ is a cube progression of size
$|Q_i| \geq N^{ c R^{-2k} }$ such that
$\diam_{Q_i}(\phi_j) \leq N^{- c R^{-2k}}$ 
for every $j \in [R]$,
and where $|\Xi| \leq N^{ d - c R^{-2k} }$.
\end{proposition}

\begin{proof}
We induct on $k \geq 0$ ;
when $k=0$ all the polynomials are zero and 
we can take $Q_1 = [N]^d$ and $\Xi = \varnothing$.
We now assume that $k \geq 1$,
and throughout the proof
we let implicit or explicit constants depend at most on $k$ and $d$.
The letters $c$ and $C$ denote positive such constants whose
value may change from line to line.

Let $L \geq 1$ and $q \geq 1$
be parameters to be determined later.
By partitioning $[N]^d$ into congruence classes
and then into subcubes, it is easy to find a partition of the form
$[N]^d = \bigsqcup_{\bfv \in V} ( \bfv + q [L]^d ) \bigsqcup \Xi$
with $|\Xi| \lesssim N^{d-1/2}$, as long as $qL \leq N^{1/2}$.
Consider an index $j \in [R]$
and the Taylor expansion of $\phi_j$ at $\bfv \in V$ given by
\begin{align*}
	\phi_j( \bfv + q \bfx ) 
	= \sum_{1 \leq |\bfalpha| \leq k} \frac{\partial^{\bfalpha} \phi_j( \bfv )}{\bfalpha !} q^{|\bfalpha|} \bfx^{\bfalpha}
	= \sum_{|\bfalpha| = k} q^k \theta_{\bfalpha,j} \bfx^{\bfalpha}
		+ \psi_{\bfv,q,j}(\bfx),
\end{align*}
where $\bfx \in \Z^d$, $\theta_{\bfalpha,j} \in \R$ 
and every $\psi_{\bfv,q,j} \in \R[x_1,\dots,x_d]$
has degree less than $k$ and zero constant coefficient
(since $\phi_j$ has degree at most $k$, its derivatives of order $k$ are constant).
Consequently we have, for every $j \in [R]$, $\bfv \in V$, $\bfx,\bfy \in \Z^d$,
\begin{align*}
	\phi_j( \bfv + q \bfx ) - \phi_j( \bfv + q \bfy )
	= \sum_{|\bfalpha| = k}
	q^k \theta_{\bfalpha,j} ( \bfx^{\bfalpha} - \bfy^{\bfalpha} )
	+ \psi_{\bfv,q,j}( \bfx ) - \psi_{\bfv,q,j}( \bfy ).
\end{align*}
When $\bfx, \bfy \in [L]^d$, by
the triangle inequality for the distance on $\T$, this implies that
\begin{align}
\label{eq:add:Phi}
	\| 	\phi_j( \bfv + q \bfx ) - \phi_j( \bfv + q \bfy ) \|_\T
	\lesssim L^{k} \max_{|\bfalpha| = k} \| q^k \theta_{\bfalpha,j} \|_\T
	 + \| \psi_{\bfv,q,j}( \bfx ) - \psi_{\bfv,q,j}( \bfy ) \|_\T.
\end{align}

At this point we use Proposition~\ref{thm:add:SimultDiophApprox}
to pick $1 \leq q \leq N^{1/4}$ 
such that $\| q^k \theta_{\bfalpha,j} \|_\T \leq N^{-cR^{-2}}$
for every $j \in [R]$ and every $|\bfalpha| = k$,
which is possible for $N \geq (2R)^{C R^2}$.
For each fixed $\bfv \in V$, we assume that $L \geq (2R)^{CR^{2(k-1)}}$
and use the induction hypothesis
to obtain a partition 
$[ L ]^d = (\, \bigsqcup_{\bfw \in W} Q_{\bfv,\bfw} ) \bigsqcup \Xi_{\bfv}$,
where each $Q_{\bfv,\bfw}$ is a cube progression such that
$|Q_{\bfv,\bfw}| \geq L^{cR^{-2(k-1)}}$ and
$\diam_{Q_{\bfv,\bfw}}( \psi_{\bfv,q,j} ) \leq L^{- c R^{-2(k-1)}}$
for every $j \in [R]$,
and with $|\Xi_{\bfv}| \leq L^{d - c R^{-2(k-1)}}$.
Inserting these diophantine and diameter bounds 
into~\eqref{eq:add:Phi}, we obtain
\begin{align}
\label{eq:add:Phi2}
	\| 	\phi_j( \bfv + q \bfx ) - \phi_j( \bfv + q \bfy ) \|_\T
	\lesssim L^k N^{-cR^{-2}} + L^{-cR^{-2(k-1)}},
\end{align}
uniformly for $j \in [R]$, $\bfv \in V$ and $\bfx,\bfy \in [L]^d$.

We choose finally $L = N^{c' R^{-2}}$ with $c'$ small enough so that
$L \leq N^{1/4}$ and the right-hand side of~\eqref{eq:add:Phi2} is
$O( N^{-c R^{-2k}} )$.
Working back through the conditions on $L$, we find
that this requires $N \geq (2R)^{CR^{2k}}$,
and when $C$ is large enough
we have therefore $\diam_{\bfv + q Q_{\bfv,\bfw}} \phi_j \leq N^{-cR^{-2k}}$
for all $j,\bfv,\bfw$.
We obtain a partition
\begin{align*}
	\textstyle
	[N]^d = \bigsqcup\limits_{\substack{ \bfv \in V \\ \bfw \in W }} ( \bfv + q Q_{\bfv,\bfw} )
			\,\bigsqcup\,
			\bigsqcup\limits_{\bfv \in V} ( \bfv + q\Xi_\bfv ) \bigsqcup \Xi.
\end{align*}
%Since the $\Xi_\bfv$ are disjoint subsets of $[N]^d \smallsetminus \Xi$
Since each set $\bfv + q \Xi_{\bfv}$ has density at most
$N^{-cR^{-2k}}$ in its ambient box $\bfv + q [L]^d$, the
disjoint union $\Xi' = \bigsqcup_{\bfv \in V} ( \bfv + q \Xi_\bfv )$ contained in $[N]^d$ 
has size at most $N^{d-cR^{-2k}}$,
and $\Xi'' = \Xi' \bigsqcup \Xi$ has size at most $N^{d-c'R^{-2k}}$.
\end{proof}

To proceed further we need to recall the language of 
factors~\cite[Section~6]{Tao:Factors},
a specialization of the theory of 
conditional expectations~\cite[Chapter~7]{Galambos:Book} to the finite setting.
We call factor a $\sigma$-algebra of the finite set $[N]^d$. 
It can be verified that the factors of $[N]^d$ are in 
one-to-one correspondence with its partitions via
\begin{align}
\label{eq:add:FactorsCorresp}
	\textstyle
	(B_i)_{i \in [\ell]} \ \ \text{such that} \ \ [N]^d = \bigsqcup\limits_{i=1}^\ell B_i
	\quad\mapsto\quad
	\calB = \big\lbrace\, \bigsqcup\limits_{i \in J} B_i,\ J \subset [\ell] \,\big\rbrace.
\end{align}
%and in that situation we say that $\calB$ is generated by the sets $B_i$.
We define an atom of a factor $\calB$ as a minimal non-empty
element of $\calB$, and those are the sets $B_i$ under the 
correspondence~\eqref{eq:add:FactorsCorresp}.
%We write $\calB(x)$ for the unique atom containing an element $x \in [N]^d$.
It can be verified that $f : [N]^d \rightarrow \C$ is
$\calB$-measurable if and only if it is constant on every atom of $\calB$.
We define the full factor $\calB_{\mathrm{full}}$ as the factor
whose atoms are all the singletons of $[N]^d$,
so that every function $f : [N] \rightarrow \C$ is $\calB_{\mathrm{full}}$-measurable,
%so that $\calB_{\mathrm{full}}$ is the set of subsets of $[N]^d$.
%Viewing (as we may) a function $f : [N]^d \rightarrow \C$ as
%a random variable on $([N]^d,\calB_{\mathrm{full}})$,
and has a well-defined conditional expectation 
$\EB{f}$ for any factor $\calB$ of $[N]^d$.
One can check that 
%$\EB{f}(x) = \E_{\calB(x)} f$ for every $x \in [N]^d$,
$\EB{f} = \sum_{i \in [\ell]} (\E_{B_i} f) 1_{B_i}$
under the correspondence~\eqref{eq:add:FactorsCorresp}.
All the usual properties of conditional expectation can be verified
directly in the finite setting, and we encourage the reader to do so as needed.

In our situation, the language of factors will serve to simplify the
step~\cite[Section~9]{me:logkeil} 
of the energy-increment strategy where the balanced function 
is replaced by an averaged version of itself over
a family of arithmetic progressions, 
which we now interpret as a conditional expectation.
The function $g$ below corresponds to the function $f_A/\delta$ of
Proposition~\ref{thm:add:LargeRestrictedMoment},
and when $\Xi$ is a subset of $[N]^d$ we write $\Xi^c = [N]^d \smallsetminus \Xi$.

\begin{proposition}[Conditioning the balanced function]
\label{thm:add:LargeMomentConditioned}
Let $\delta \in (0,1]$ and suppose that $g : [N]^d \rightarrow \C$
is such that $\| g \|_\infty \leq \delta^{-1}$.
Suppose that, for certain constants $c_1,C_1 > 0$,
there exist $1 \leq R \leq (\delta/2)^{-C_1}$
and distinct frequencies $\bfxi_1\dots,\bfxi_R \in \ZM$ 
such that
\begin{align}
\label{eq:add:LargeRestrictedMoment}
	R^{c_1} \lesssim \sum_{i=1}^R |H_g(\bfxi_i)|^{s'}.
\end{align}
Then there exists $C_2 > 0$ such that, 
when $N \geq e^{(\delta/2)^{-C_2}}$, the following holds.
Consider the polynomial phase functions $\phi_1,\dots,\phi_R : \Z^d \rightarrow \T$ such that 
\begin{align*}
	&\phantom{(1 \leq i \leq r),} &
	H_g(\bfxi_i) &= \E_{n \in [N]^d} g(\bfn) e(\phi_i(\bfn))
	&&(1 \leq i \leq R),
\end{align*}
and consider the partition $[N]^d = ( \bigsqcup_i Q_i ) \bigsqcup \Xi$ 
given by Proposition~\ref{thm:add:LinPolPhases}.
Let $\calB$ be the factor of $[N]^d$ 
corresponding to this partition,
%whose atoms are $\Xi$ and the $Q_i$,
and write $\wt{g} = \EB{ g 1_{\Xi^c} }$.
Then
\begin{align*}
	R^{c_1} \lesssim
	\| H_{\wt{g}} \|_{s'}^{s'}.
\end{align*}
\end{proposition}

\begin{proof}
Consider an index $i \in [R]$.
We first neglect the error set $\Xi$ via
\begin{align}
\label{eq:add:NeglectXi}
	H_g(\bfxi_i) = \E\big[ g e(\phi_i) \big]
				= \E\big[ g 1_{\Xi^c} e(\phi_i) \big] + O(\delta^{-1}N^{-cR^{-2k}}).
\end{align}
Since $1_{\Xi^c} e(\phi_i)$ is almost constant 
on each cube progression $Q_j$ and zero on $\Xi$, we have
\begin{align*}
%	\EB{ g 1_{\Xi^c} e(\phi_i) }(x)
%	&= \E_{y \in \calB(x)} g(y) 1_{\Xi^c}(y) e(\phi_i(y)) \\
%	&= \E_{y \in \calB(x)} g(y) 1_{\Xi^c}(y) e(\phi_i(x)) + O(\delta^{-1}N^{-cR^{-2k}}) \\
%	&= \EB{g 1_{\Xi^c}}(x) e(\phi_i(x)) + O(\delta^{-1}N^{-cR^{-2k}}).
	\EB{ g 1_{\Xi^c} e(\phi_i) }
	= \EB{ g 1_{\Xi^c} } e(\phi_i) + O(\delta^{-1}N^{-cR^{-2k}}).
\end{align*}
Returning to~\eqref{eq:add:NeglectXi}, we can exploit this fact by conditioning on $\calB$ in
\begin{align*}
	H_g(\bfxi_i) 
	&= \E\Big[ \EB{ g 1_{\Xi^c} e(\phi_i) } \Big] + O(\delta^{-1}N^{-cR^{-2k}}) \\
	&= \E\Big[ \EB{ g 1_{\Xi^c} } e(\phi_i) \Big] + O(\delta^{-1}N^{-cR^{-2k}}) \\
	&= H_{ \wt{g} }( \bfxi_i ) + O(\delta^{-1}N^{-cR^{-2k}}).
\end{align*}
We can insert this estimate in~\eqref{eq:add:LargeRestrictedMoment} to obtain
\begin{align*}
	R^{c_1} \lesssim \sum_{i=1}^R |H_{\wt{g}}(\bfxi_i)|^{s'}
	+ O\big( R ( \delta^{-1}N^{-cR^{-2k}} )^{s'} \big).
\end{align*}
Recalling the size condition on $R$, and completing the sum,
we obtain the desired statement when $N \geq e^{(\delta/2)^{-C_2}}$
with $C_2 > 0$ large enough.
\end{proof}

Using the previous proposition and restriction estimates,
we aim to obtain a lower bound on the energy of the conditioned balanced function.
If we succeed in doing so, the
following proposition then yields a density increment.

\begin{proposition}[$L^2$ density increment]
\label{thm:add:L2DensIncr}
Let $\kappa \in [c_3,+\infty)$ for a constant $c_3 > 0$.
Suppose that $\calB$ is a factor of $[N]^d$ with atoms $(Q_i)$, $\Xi$
such that $|\Xi| \leq N^{d - (\delta/2)^{C_3}}$ for a constant $C_3 > 0$.
Suppose also that $A$ is a subset of $[N]^d$ of density $\delta$ such that
\begin{align*}
	\kappa \delta 
	\leq \| \EB{ f_A 1_{\Xi^c} } \|_{L^2[N]}.
\end{align*}
Then there exists $C_4 > 0$ such that,
for $N \geq e^{-(\delta/2)^{-C_4}}$,
there exists an atom $Q_i$ with
\begin{align*}
	( 1 + \tfrac{1}{2}\kappa^2 ) \delta \leq \frac{|A \cap Q_i|}{|Q_i|}.
\end{align*}
\end{proposition}

\begin{proof}
First note that $\EB{ 1_{[N]^d} 1_{\Xi^c} } = 1_{[N]^d \smallsetminus \Xi}$.
We write $\| \cdot \|_2 = \| \cdot \|_{L^2[N]}$ throughout this proof.
Expanding the square, we obtain
\begin{align*}
	\kappa^2 \delta^2
	&\leq \| \EB{ 1_{A \smallsetminus \Xi} } - \delta 1_{[N]^d \smallsetminus \Xi} \|_2^2	\\
	&\leq \| \EB{1_{ A \smallsetminus \Xi} } \|_2^2 - 2 \delta \langle \EB{1_{A \smallsetminus \Xi}} , 1_{[N]^d \smallsetminus \Xi} \rangle 
	+ \delta^2 \| 1_{[N]^d \smallsetminus \Xi} \|_2^2.
\end{align*}
Let $A' = A \smallsetminus \Xi$.
Since the conditional expectation operator is self-adjoint, we have then
\begin{align*}
	\kappa^2 \delta^2
	&\leq \| \EB{1_{A'}} \|_2^2 - 2 \delta \langle 1_{A'} , \EB{1_{[N]^d \smallsetminus \Xi}} \rangle + \delta^2 
	+ O( N^{ -(\delta/2)^{C_3} } )	
	\\
	&= \| \EB{1_{A'}} \|_2^2 - \delta^2 + O( N^{ -(\delta/2)^{C_3} } ).
\end{align*}
Assuming that $N \geq e^{-(\delta/2)^{-C_4}}$ with $C_4 > 0$ large enough, we have
\begin{align*}
	(1 + \tfrac{1}{2}\kappa^2) \delta^2 
	&\leq \| \EB{ 1_{A'} } \|_2^2 \\
	&\leq \| \EB{ 1_{A'} } \|_\infty \cdot \E\Big[ \EB{ 1_{A'} } \Big] \\
	&\leq \max_i ( \E_{Q_i} 1_A ) \cdot \delta, 
\end{align*}
where we have ignored the $\Xi$-average since $\E_{\Xi} 1_{A'} = 0$.
This gives the desired conclusion upon dividing by $\delta$.
\end{proof}

We are finally ready to derive our main iterative proposition.
It is at this point that we genuinely exploit the two types
of restriction estimates of Proposition~\ref{thm:add:HBounds},
in order to first obtain a lower bound on the energy of the 
conditioned balanced function,
and then apply a complete $L^2 \rightarrow L^p$ estimate.
At this stage we may also reduce our working hypothesis
to $A$ not containing any non-trivial solutions,
by our assumption~\eqref{eq:intro:NumberSolsBounds}
and the fact that $N$ is already assumed to be quite large
with respect to the density $\delta$.

\begin{proposition}
\label{thm:add:FinalDensIncr}
There exist positive constants $c,C$ such that the following holds.
Suppose that $A$ is a subset of $[N]^d$ of density $\delta$
%such that $T(1_A,\dots,1_A) \leq c_0 \delta^s$,
such that all solutions $(\bfn_i) \in A^s$
to~\eqref{eq:add:SystPols} lie in $Z$,
and that $N \geq e^{-(\delta/2)^{-C}}$.
Then there exists $1 \leq R \leq (\delta/2)^{-C}$
and a cube progression $Q \subset [N]^d$ of size $N'$ such that,
writing $\delta' = |A \cap Q| / |Q|$, we have
\begin{align*}
	\delta' \geq ( 1 + c R^c ) \cdot \delta,
	\qquad
	N' \geq N^{c R^{-2k}}.
\end{align*}
\end{proposition}

\begin{proof}
In the context of this proof, we let
$c,C$ denote positive constants whose value
may change from line to line, and which may
depend on $\bfP$ and $\bflambda$ as usual.
Since all solutions $(\bfn_i) \in A^s$ to~\eqref{eq:add:SystPols} lie in $Z$,
it follows from~\eqref{eq:add:TDef} and~\eqref{eq:intro:NumberSolsBounds} that
\begin{align*}
	T(1_A,\dots,1_A) \leq C N^{-\omega} \leq c_0 \delta^s,
\end{align*}
for $N \geq C\delta^{-s/\omega}$, where $c_0$ is the constant
in Proposition~\ref{thm:add:LargeRestrictedMoment}.
Assuming furthermore that $N \geq e^{(\delta/2)^{-C}}$ 
for a large enough $C > 0$,
we can then combine Propositions~\ref{thm:add:LargeRestrictedMoment}
and~\ref{thm:add:LargeMomentConditioned}
to obtain $1 \leq R \leq (\delta/2)^{-C}$ such that
\begin{align}
\label{eq:add:LargeMomentConditioned}
	\delta R^c \lesssim \| H_{ \wt{f}_A } \|_{s'},
\end{align}
where $\wt{f}_A = \EB{ f_A 1_{\Xi^c} }$ and $\calB$ is a factor of $[N]^d$
generated by atoms $(Q_i),\Xi$, with each $Q_i$ 
being a cube progression with $|Q_i| \geq N^{cR^{-C}}$
and with $|\Xi| \leq N^{d - (\delta/2)^C}$.
From~\eqref{eq:add:LargeMomentConditioned} 
and~\eqref{eq:add:HBoundLinfty}, 
noting also that $\| \wt{f}_A \|_\infty \leq \| f_A \|_\infty \leq 1$,
we deduce that for some $C > 0$,
\begin{align*}
	\delta^C \lesssim \| \wt{f}_A \|_{L^2[N]}.
\end{align*}
%for a certain constant $C > 0$.
By assuming that $\delta \geq N^{-c}$ with $c > 0$ small enough and $N$ large,
%$\delta \geq N^{-\nu/2C}$, where $\nu$ is the constant in Proposition~\ref{thm:add:HBounds},
we can ensure that $N^{-\nu} \leq \| \wt{f}_A \|_{L^2[N]}$,
where $\nu$ is the constant from Proposition~\ref{thm:add:HBounds},
and on the other hand we have $\| \wt{f}_A \|_\infty \leq 1$.
We may therefore apply~\eqref{eq:add:HBoundL2}
in~\eqref{eq:add:LargeMomentConditioned} to obtain
\begin{align*}
	\delta R^c \lesssim \| \EB{ f_A 1_{\Xi^c} } \|_{L^2[N]}.
\end{align*}
At this stage we can simply apply Proposition~\ref{thm:add:L2DensIncr}
to obtain the coveted density increment.
\end{proof}

The proof of Theorem~\ref{thm:intro:SystTslInv}
now follows by an iteration entirely similar to the one
in the one-dimensional setting~\cite[Section~4]{me:logkeil}.
%We indicate the few necessary modifications below.
%We make an exception to our convention in letting N not be fixed...
\smallskip

\textit{Proof of Theorem~\ref{thm:intro:SystTslInv}.}
%We construct iteratively a sequence of integers $N_i$ 
%and a sequence of subsets $A_i$ of $[N_i]^d$ of density $\delta_i$,
%initializing at $(A_0,N_0,\delta_0) = (A,N,\delta)$.
It suffices to follow
the proof of~\cite[Theorem~2]{me:logkeil}
in~\cite[Section~4]{me:logkeil}, 
\textit{mutadis mutandis},
replacing~\cite[Proposition~4.1]{me:logkeil} 
by Proposition~\ref{thm:add:FinalDensIncr},
arithmetic progressions by cube progressions,
and trivial solutions by the set $Z$.
The powers of $R$ differ in the two cases but this does not
affect the final bound.
Since the constants in the statement of
Proposition~\ref{thm:add:FinalDensIncr} 
were allowed to depend on $\bfP$, $\bflambda$, 
the final logarithm exponent now depends on these parameters as well.
When the algorithm stops, one obtains a cube progression
$Q = \bfv + q [L]^d$ with $\bfv \in \Z^d$ and $q \geq 1$
such that, if we write $A \cap Q = \bfv + qA'$,
there exists $(\bfn_i) \in (A')^s \smallsetminus Z$ satisfying~\eqref{eq:add:SystPols}. 
By translation-dilation invariance of $Z$ and of~\eqref{eq:add:SystPols},
it follows that $(\bfv + q\bfn_i) \in A^s \smallsetminus Z$ 
also satisfies~\eqref{eq:add:SystPols}, and the proof is complete.
\qed

\section{On epsilon-removal}
\label{sec:remov}

We fix an integer $N \geq 1$ to be thought of as large,
and an integer $k \geq 3$.
We write
\begin{align*}
	\Gamma = \{ (n,\dots,n^k),\ 1 \leq n \leq N \},
	\qquad \dsigma_\Gamma = 1_{\Gamma} \dSigma.
\end{align*}
We define the corresponding Weyl sum
\begin{align*}
	&\phantom{(\bfalpha \in \T^k)} &
	F( \bfalpha )
	&= \sum_{n \leq N} e( \alpha_1 n + \dotsb + \alpha_k n^k)
	&&(\bfalpha \in \T^k).
\end{align*}
Given a weight function $g : \Z^k \rightarrow \C$, we also define\footnote{
Note that $F_g = F_a^{(x,\dots,x^k)}$ with $a(n) = g(n,\dots,n^k)$
in the notation of the introduction, but this new definition
is more natural from a Fourier-analytic point of view.}
\begin{align}
\label{eq:remov:ExpSumWeighted}
	&\phantom{(\bfalpha \in \T^k)} &
	F_g(\bfalpha) &= \sum_{\bfn \in \Gamma} g( \mathbf{n} ) e( \bfalpha \cdot \bfn ) 
	= (g \dsigma_\Gamma)^\wedge( \bfalpha )
	&&(\bfalpha \in \T^k).
\end{align}
so that $F = (\dsigma_\Gamma)^\wedge$ in the unweighted case $g \equiv 1$.
The goal of this section is to prove an estimate of the form~\eqref{eq:intro:RestrTrunc}
for $\bfP=(x,\dots,x^k)$,
by a modification of the argument of Bourgain~\cite{Bourgain:Squares} for squares.
Hughes was the first to obtain results in this direction
%for the curve $(x,\dots,x^k)$
in unpublished work from 2013.
We include our alternative argument\footnote{
Very recently, Wooley~\cite{Wooley:Restr} has independently obtained a similar estimate.
%, in fact with a slightly improved range of exponents
} for two main reasons:
to illustrate the philosophy that truncated restriction estimates
are simpler to obtain than full ones, requiring as they do 
only major arc information on
unweighted exponential sums, and also to show how
these estimates naturally extend to the multidimensional setting.

\begin{proposition}[Truncated restriction estimate for monomial curves]
\label{thm:remov:TruncRestr}
Let $k \geq 3$ and write $K = \tfrac{1}{2}k(k+1)$.
Let $\theta = 1/12$ if $k=3$, and $\theta = \max( 2^{-k}, 1 / 8 s_{k-1} )$ else.
%Let $m_k$ be the smallest even integer $m$ such that $m > K + 2$.
Then, for every $\eps > 0$,
\begin{align*}
	\int\limits_{\textstyle |F_g| \geq N^{-\theta + \eps + 1/2} \| g \|_2 }	|F_g(\bfalpha)|^p 
	\ \dbfalpha
	\ \lesssim_{p,\eps}\ N^{\tfrac{p}{2} - K} \| g \|_2^p
	\qquad\text{for $p > 2K + 4$}.
\end{align*}
\end{proposition}

We refer to Definition~\ref{thm:intro:skDef} 
for the meaning of $s_k$.
We pay attention to the quality of the exponent $\theta$ above,
although this is not necessary for our applications,
and the proof could be simplified slightly by ignoring this aspect.
The previous proposition has the following more familiar consequence,
which again is not strictly required for our later argument.

\begin{corollary}[$\eps$-removal for monomial curves]
\label{thm:remov:EpsRemoval}
Let $k \geq 4$ and write $K = \tfrac{1}{2}k(k+1)$.
Suppose that, for some $q > 0$,
\begin{align*}
	\int_{\T^k} |F_g|^q \ \dm \lesssim_\eps N^{\tfrac{q}{2} - K + \eps} \| g \|_2^q
\end{align*}
for every $\eps > 0$. Then, for $p > max(2K+4,q)$,
\begin{align*}
	\int_{\T^k} |F_g|^p \ \dm \lesssim N^{\tfrac{p}{2} - K} \| g \|_2^p.
\end{align*}
\end{corollary}

\begin{proof}
Without loss of generality we may assume that $\| g \|_2 = 1$.
By Proposition~\ref{thm:remov:TruncRestr}, it suffices to bound the tail
\begin{align*}
	\int_{ |F_g| \leq N^{-\theta + \eps + 1/2} } |F_g|^p \ \dm
	&\leq N^{ - (p-q) ( \theta - \eps ) } N^{(p-q)/2} \int_{\T^k} |F_g|^q \ \dm
	\\
	&\lesssim_{\eps} N^{\eps - (p-q)(\theta - \eps)} N^{p/2 - K}
	\\
	&\lesssim N^{p/2 - K}. \phantom{\int_0^1}
\end{align*}
\end{proof}

We start by recalling the basics of the discrete Tomas-Stein 
argument~\cite{Bourgain:Squares,Bourgain:ParabI}.
We fix a function $g : \Z^d \rightarrow \C$,
and for a parameter $\eta > 0$ we define
\begin{align*}
	E_\eta = \{ |F_g| \geq \eta N^{1/2} \},
	\qquad
	f_0 = 1_{E_\eta} \frac{F_g}{|F_g|},
	\qquad
	f = 1_{E_\eta}.
\end{align*}
We assume that $\| g \|_2 = 1$ throughout,
so that $|F_g| \leq N^{1/2}$ by Cauchy-Schwarz in~\eqref{eq:remov:ExpSumWeighted}, 
and we can assume that $\eta$ lies in $(0,1]$.
We will bound the moments of $F_g$ of order $p \geq 1$ 
through the formula
\begin{align}
\label{eq:remov:MomentsViaLevelSets}
	\int_{a N^{1/2}}^{b N^{1/2}} |F_g|^p \ \dm
	= p N^{p/2} \int_a^b \eta^{p-1} |E_\eta| \deta
	\qquad
	\text{for}
	\qquad 0 \leq a \leq b \leq 1.
\end{align}
By definition of $f_0$ and Parseval, we have
\begin{align*}
	\eta N^{1/2} |E_\eta| 
	\leq \langle f_0 , F_g \rangle
	= \langle f_0 , ( g \dsigma_\Gamma )^\wedge \rangle
	= \langle \wh{f}_0 , g \rangle_{L^2(\dsigma_\Gamma)}.
\end{align*}
By Cauchy-Schwarz and using the assumption $\|g\|_2 = 1$, it follows that
\begin{align*}
	\eta^2 N |E_\eta|^2
	\leq \| \wh{f}_0 \|_{L^2(\dsigma_\Gamma)}^2
	= \langle \wh{f}_0 \dsigma_\Gamma , \wh{f}_0 \rangle.
\end{align*}
By another application of Parseval, we conclude that
\begin{align}
\label{eq:remov:TomasStein}
	\eta^2 N |E_\eta|^2 \leq \langle f_0 \ast F, f_0 \rangle.
\end{align}
This well-known inequality is the starting point
of our argument.

We now use the circle method to decompose the kernel $F$
into two pieces, corresponding to the usual major and minor arcs.
%We recall the following standard 
%approximation formula~\cite[Lemma~4.4]{Baker:Book}.
%\begin{proposition}
%\label{thm:weyl:ApproxFormula}
%Let $c_k = 1/2k^2$.
%Suppose that $\bfalpha = \bfa/q + \bfbeta$
%with $(\bfa,q) = 1$ and $|\beta_j| \leq c_k q^{-1} N^{1 - k_j}$
%for $1 \leq j \leq t$.
%Then
%\begin{align*}
%	F(\bfalpha) &= q^{-1} S(\bfa,q) I( \bfalpha - \bfa /q , N )
%	+ O_\eps( q^{1 - 1/k + \eps} ).
%\end{align*}
%\end{proposition}
To bound $F$ on minor arcs we will use the
following estimates of Weyl/Vinogradov type.

\begin{proposition}
%[Weyl-type estimates]
\label{thm:remov:MinArcBound}
Let $k \geq 3$ be an integer
and let $\tau,\delta$
be real numbers with $0 < \tau < \max( 2^{1-k}, 1/4s_{k-1} )$
and $\delta > k \tau$.
Then if $|F(\bfalpha)| \geq N^{1 - \tau}$
and $N$ is large enough with respect to $k,\tau,\delta$,
there exist integers $q,a_1,\dots,a_k$ 
such that $1 \leq q \leq N^\delta$,
$(a_1,\dots,a_k,q) = 1$ and
$| q \alpha_j - a_j | \leq N^{\delta - k_j}$ for $1 \leq j \leq k$.
\end{proposition}

\begin{proof}
When $\tau = 2^{1-k}$, this is~\cite[Theorem~5.1]{Baker:Book},
with parameters $M = 1$, $P = N^{1 - \tau}$
and choosing the $\eps$ from that theorem small enough
so that $k\tau + \eps \leq \delta$.
When $\tau = 1/4s_{k-1}$, the proposition follows
from the reasoning used in the proof of~\cite[Theorem~1.6]{Wooley:EffI}
in~\cite[Section~8]{Wooley:EffI}.
\end{proof}

We adopt the convention that
any implicit or explicit constant throughout the
section may depend on $k$,
and we assume that $N$
is large enough with respect to $k$
when needed by the argument, 
without further indication.
(Since $\| F_a \|_\infty \leq N$, we may certainly assume that $N$ is larger
than any absolute constant in proving 
Proposition~\ref{thm:remov:TruncRestr}).
We set $\tau = \frac{1}{6}$ if $k=3$ and 
$\tau = \max( 2^{1-k}, 1/4s_{k-1} )$ if $k \geq 4$,
in accordance with the Weyl-type estimates we intend to use.
We fix a small quantity $\eps_0 \in (0,\tau)$
and a constant $\delta = k(\tau - \eps_0)$.
For $k \geq 4$, we can use the bound 
$s_{k-1} \geq \tfrac{1}{2}k(k-1)$ to deduce that
\begin{align*}
	\delta	
	< k\tau
	\leq \max\Big( \frac{k}{2^{k-1}}, \frac{k}{4s_{k-1}} \Big) 
	\leq \max\Big( \frac{k}{2^{k-1}} , \frac{1}{2(k-1)} \Big) 
	\leq \frac{1}{2},
\end{align*}
and the same bound holds for $k=3$ trivially.
We define the major and minor arcs in a standard fashion by
\begin{align}
	\notag
	\frakM(\bfa,q) &= \{ \bfalpha \in \T^k : 
	\| \alpha_j - a_j/q \| \leq q^{-1} N^{\delta - j} 
	\ (1 \leq j \leq k) \}, 
	\\
	\label{eq:remov:MajorArcs}
	\frakM &=	\bigsqcup_{ q \leq N^\delta } \bigsqcup_{\substack{ \bfa \in [q]^k \,: \\ (\bfa,q) = 1 }} \frakM(\bfa,q),
	\quad\quad \frakm = \T^k \smallsetminus \frakM.
\end{align}
It is easy to check that we have indeed a 
disjoint union in~\eqref{eq:remov:MajorArcs} when $\delta < 1/2$.
We use the fundamental domain $\frakU = (\tfrac{1}{2} N^{-\delta}, 1 + \tfrac{1}{2} N^{-\delta}]^k$
containing the intervals $\bfa/q + \prod_j [-q^{-1}N^{\delta-j},q^{-1} N^{\delta-j}]$
with $1 \leq q \leq N^\delta$ and $\bfa \in [q]^k$.
%if we were to have an non-empty intersection
%$\frakM(\bfa,q) \cap \frakM(\bfa',q')$
%with $q,q' \leq X^\delta$, $(\bfa,q) = (\bfa',q') = 1$
%and $(\bfa,q) \neq (\bfa,q')$,
%then there would exist an index $j \in [k]$ 
%such that $a_j/q \neq a_j'/q'$ and
%\begin{align*}
%	\frac{1}{qq'} \leq \| a_j / q - a'_j / q' \|_\T 
%	\leq \Big( \frac{1}{q} + \frac{1}{q'} \Big) N^{\delta - j}.
%\end{align*}
%This inequality is easily seen 
%to yield a contradiction when $\delta < 1/2$.

We first obtain a set of estimates for
the exponential sum $F$ on minor and major arcs.
This involves the Gaussian sum and oscillatory integral defined respectively by
%to the curve $\Gamma$, defined respectively by
\begin{align}
%	\label{eq:remov:GaussSum}
	\notag
	&\phantom{(a \in \Z_q^k)} &
	S( \bfa , q ) &= \sum_{u \bmod q} e_q ( a_1 u + \dotsb + a_k u^k )
	&&(\bfa \in \Z_q^k),
	\\
	\label{eq:remov:OscIntg}
	&\phantom{(\bfbeta \in \R)} &
	I( \bfbeta , N ) &= \int_0^N e( \beta_1 x + \dotsb + \beta_k x^k ) \dx
	&&(\bfbeta \in \R^k).
\end{align}

\begin{proposition}
\label{thm:remov:WeylSumBounds}
%	For $N \geq C_{\eps_0}$, 
	For $\bfalpha \in \frakU$, we have
	\begin{align*}
%	\label{eq:remov:WeylSumBounds}
		|F( \bfalpha )| 
		= \begin{cases}
		O_{\eps_0}( N^{ 1 - \tau + 2\eps_0 } )
		& \text{if $\bfalpha \in \frakm$}, 
		\\
		q^{-1} S(\bfa,q) I( \bfalpha - \bfa /q , N )
		+ O_{\eps_0}( N^{ 1 - \tau + 2\eps_0 } )
		& \text{if $\bfalpha \in \frakM(\bfa,q) \subset \frakM$}.
		\end{cases}
	\end{align*}
\end{proposition}

\begin{proof}
Consider a frequency $\bfalpha \in \T^k$.
If $|F(\bfalpha)| \geq N^{1 - (\tau - 2\eps_0)}$ and $N$ is large enough,
then Proposition~\ref{thm:remov:MinArcBound} with $\tau \leftarrow \tau - 2\eps_0$ and
$\delta \leftarrow k(\tau - \eps_0)$ shows that $\bfalpha \in \frakM$.
Therefore $|F| \lesssim_{\eps_0} N^{1 - \tau + 2\eps_0}$ on $\frakm$.
%for else (...) shows that $\bfalpha \in \frakM$ 
%(upon dividing the coefficients $a_i,q$ from there by $(\bfa,q)$).

When $\bfalpha \in \frakM(\bfa,q)$
with $1 \leq q \leq N^\delta$,
$\bfa \in [q]^k$ and $(\bfa,q) = 1$, we have, 
for every $j \in [k]$,
\begin{align*}
	| \alpha_j - a_j / q |
	\leq q^{-1} N^{\delta - j}
	\leq (2k^2)^{-1} q^{-1} N^{1 - j},	
\end{align*}
where we used the fact that $\delta < 1$ and $N$ is large
in the last inequality.
By a standard Poisson-based approximation formula~\cite[Lemma~4.4]{Baker:Book},
we obtain the desired approximation of $F$, noting that
$q^{1 - 1/k + \eps}
%\leq C_\eps N^{(1+\eps)\delta - (\tau-\eps_0)} 
\lesssim N^{1 - \tau + 2\eps_0} $
for $q \leq N^\delta$ and $\eps$ small enough.
\end{proof}

In light of the previous proposition,
we define a majorant function $U_p : \frakU \rightarrow \C$ by
\begin{align}
	\label{eq:remov:MajorantU}
	U_p = \sum_{q \leq N^\delta} \ \sum_{\substack{ \bfa \in [q]^k \,: \\ (\bfa,q) = 1 }}
	|q^{-1} S(\bfa,q)|^p \cdot 1_{\frakM(\bfa,q)} \cdot \tau_{-\bfa/q} |I( \cdot, N)|^p.
\end{align}
Our bounds on the exponential sum $F$ can be phrased
in the following form,
where we wrote $\eps = 2 \eps_0$.
%and assume that $N$ is large enough with respect
%to the quantities $\eps$.

\begin{proposition}
\label{thm:remov:KernelDcpU}
We have a decomposition $F = F_1 + F_2$ with
\begin{align*}
	\| F_2 \|_{\infty} \lesssim_\eps N^{1 - \tau + \eps}
	\quad\text{and}\quad
	|F_1|^p \leq U_p.
\end{align*}
\end{proposition}

\begin{proof}
We naturally define
\begin{align*}
	F_1 &= \sum_{q \leq N^\delta} \ \sum_{(\bfa,q) = 1} 
			q^{-1} S( \bfa, q ) \tau_{-\bfa/q} I( \,\cdot\,,N) \cdot 1_{\frakM(\bfa,q)}
%	F_2 &= \sum_{q \leq N^\delta} \ \sum_{(\bfa,q) = 1} 
%			\big( F - q^{-1} S( \bfa, q ) \tau_{-\bfa/q} I \big) 1_{\frakM(\bfa,q)} +  F 1_{\frakm_Q}.
\end{align*}
and $F_2 = F - F_1$.
Since the arcs $\frakM(\bfa,q)$ are disjoint for $q \leq N^\delta$, $(\bfa,q)=1$,
the required bounds follow from Proposition~\eqref{thm:remov:WeylSumBounds}.
\end{proof}

Our argument is a modification of Bourgain's~\cite{Bourgain:Squares}, 
in which we directly use $L^1$ bounds on the major arc majorant $U_p$
to obtain $L^\infty \rightarrow L^1$ estimates for
the operator of convolution with $U_p$.
In fact, we show that the $L^1$ norm of $U_p$ 
is controlled by the following local moments,
where we define $I(\bfbeta) = I(\bfbeta,1)$:
\begin{align}
	\label{eq:remov:SgSeriesIntg}
	\frakS_p
	= \sum_{q \geq 1} \sum_{\substack{ \bfa \in [q]^k \,: \\ (\bfa,q) = 1 }} |q^{-1}S(\bfa,q)|^p,
	\qquad
%	\label{eq:remov:SgIntg}
	\frakJ_p = \int_{\R^k} |I(\bfxi)|^p \dbfxi.
\end{align}

\begin{lemma}
\label{thm:remov:LoneBoundU}
For $p > 0$, we have
\begin{align*}
%\label{eq:remov:L1BoundU}
	\int_{\frakU} |U_p| \ \dm
	\leq \frakS_p \cdot \frakJ_p \cdot N^{p - K}.
\end{align*}
\end{lemma}

\begin{proof}
From the definition~\eqref{eq:remov:MajorantU} of $U_p$, we obtain effortlessly
\begin{align}
\label{eq:remov:LoneBoundUInterm}
	\int_{\frakU} |U_p| \ \dm
	\leq \frakS_p \cdot \int_{\R^k} |I(\bfbeta,N)|^p \dbfbeta.
\end{align}
By a linear change of variables in~\eqref{eq:remov:OscIntg}, we have
\begin{align*}
	I( \bfbeta, N )
%	&= \int_0^N e( \beta_1 x + \dotsb + \beta_k x^k ) \dx \\
	&= N \int_0^1 e( \beta_1 N x + \dotsb + \beta_k N^k x^k ) \dx \\
	&= N \cdot I( \beta_1 N , \dots , \beta_k N^k ). \phantom{\int_0^1} 
\end{align*}
By another linear change of variables, we find that
\begin{align*}
	 \int_{\R^k} |I( \bfbeta, N )|^p
	 = N^p \int_{\R^k} |I( \beta_1 N , \dots , \beta_k N^k )|^p \dbfbeta
%	 \\
	 = N^{p - K} \int_{\R^k} |I(\bfxi)|^p \dbfxi,
\end{align*}
and this can be inserted 
into~\eqref{eq:remov:LoneBoundUInterm} to finish the proof.
\end{proof}

\begin{proposition}
\label{thm:remov:LevelSetSg}
Suppose that $p > 0$ is such that $\frakS_p < \infty$ and $\frakJ_p < \infty$.
Then
\begin{align*}
	|E_\eta| \lesssim_p N^{-K} \eta^{-2p}
	\qquad\text{if $\eta \geq N^{-\tau/2 + \eps}$}
\end{align*}
when $N$ is large enough with respect to $\eps$.
\end{proposition}

\begin{proof}
Starting from the inequality~\eqref{eq:remov:TomasStein},
and using the decomposition of Proposition~\ref{thm:remov:KernelDcpU}
and Hölder's inequality,
we obtain
\begin{align*}
	\eta^2 N |E_\eta|^2
	&\leq
	\langle |F_1| \ast f , f \rangle
	+ \| F_2 \|_\infty \| f \|_1^2
	\\
	&\leq
	\| |F_1| \ast f \|_p \| f \|_{p'}
	+ 	O_\eps( N^{1 - \tau + \eps} |E_\eta|^2 ).
\end{align*}
For $\eta \geq N^{-\tau/2 + \eps}$,
applying also Young's inequality yields
\begin{align*}
	\eta^2 N |E_\eta|^2
	&\lesssim \| F_1 \|_{p} \| f \|_1 \| f \|_{p'}
	\\
	&\leq \| U_p \|_1^{1/p} |E_\eta|^{2-\frac{1}{p}},
\end{align*}
so that $|E_\lambda| \lesssim \| U_p \|_1 N^{-p} \eta^{-2p}$,
and we obtain the desired bound upon invoking
Lemma~\ref{thm:remov:LoneBoundU}.
\end{proof}

In the case of an even integer exponent $p = 2s$,
the two local moments in~\eqref{eq:remov:SgSeriesIntg}
are called respectively the singular series and the
singular integral in Tarry's problem,
and the problem of their convergence
has been solved respectively 
by Hua~\cite{Hua:SgSeriesIntg}
and Arkhipov et al.~\cite{ACK:SgIntg}.
The following is~\cite[Theorems~1.3 and~2.4]{ACK:Book},
and the method of proof used there
allows in fact for real exponents $p$.

%Note that $\frakS_p$ converges already for $p > (t+1)k$ as a consequence
%of Hua's bound~\eqref{eq:weyl:HuaBound},
%while $\frakJ_p$ converges for $p > tk$ by~\eqref{eq:weyl:OscIntgBound},
%but it is possible to go much further.

\begin{proposition}
\label{thm:remov:CvSgSeriesIntg}
Let $p > 0$, $k \geq 2$ and $K = \tfrac{1}{2}k(k+1)$.
The singular integral $\frakJ_p$ 
%in~\eqref{eq:remov:SgIntg}
converges for $p > K + 1$,
and the singular series $\frakS_p$
%in~\eqref{eq:remov:SgSeries}
converges for $p > K + 2$.
\end{proposition}

In fact, the restriction estimates of Drury~\cite{Drury:RestrCurves}
for curves yield a distinct proof of the convergence of the singular integral.
We now have all the ingredients needed to derive
a truncated restriction estimate.
\smallskip

\textit{Proof of Proposition~\ref{thm:remov:TruncRestr}.}
Let $\theta = \tau/2$ and $\nu > 0$.
Using the integration formula~\eqref{eq:remov:MomentsViaLevelSets},
and invoking Proposition~\ref{thm:remov:LevelSetSg} with $p \leftarrow K+2+\nu$
and Proposition~\ref{thm:remov:CvSgSeriesIntg}, 
we obtain
\begin{align*}
	\int_{|F_g| \geq N^{-\theta + \eps + 1/2}} |F_g|^p \ \dm
	&\asymp_p N^{p/2} \int_{N^{-\theta + \eps}}^1 \eta^{p-1} |E_\eta| \deta \\
	&\lesssim_p N^{p/2 - K} \int_0^1 \eta^{p - 2(K+2+\nu) - 1} \deta.
\end{align*}
This last quantity is $O_p(N^{p/2 - K})$ for $p > 2K + 4$ and $\nu$ small enough.
\qed
\smallskip

We comment briefly on how the $\eps$-removal lemma
we have just proven
extends to the multidimensional setting.
Since we only need major arc information and
any inequality of Weyl type, we rely essentially
on work of Arkhipov et al.~\cite{ACK:Book}
from the decade 1970--1980.
We pick a finite subset $E$ of $\N_0^d \smallsetminus \{0\}$ 
and consider the set
\begin{align*}
	S = \{\, ( n_1^{j_1} \cdots n_d^{j_d} )_{ (j_1,\dots,j_d) \in E } \,:\, n_1,\dots,n_d \in [N] \,\}
\end{align*}
corresponding to the reduced system of polynomials
$\bfP = (\bfx^{\bfj},\,\bfj \in E)$
of degree $k = \max_{\bfj \in E} |\bfj|$ and rank $r = |E|$.
%\begin{align*}
%	&\phantom{(\bfj \in E)} &
%	\bfx_1^{\bfj} + \dotsb + \bfx_s^{\bfj}
%	&= \bfy_1^{\bfj} + \dotsb + \bfy_s^{\bfj}
%	&&(\bfj \in E).
%\end{align*}
The exponential sums~\eqref{eq:intro:ExpSums} become
\begin{align}
\label{eq:remov:WeylSumMultidim}
	F_a^{(\bfP)}( \bfalpha ) 
	= \sum_{ \bfn \in [N]^d } a(\bfn) e\bigg( \sum_{\bfj \in E} \alpha_{\bfj} \bfn^{\bfj} \bigg),
	\quad
	F^{(\bfP)}( \bfalpha ) 
	= \sum_{ \bfn \in [N]^d } e\bigg( \sum_{\bfj \in E} \alpha_{\bfj} \bfn^{\bfj} \bigg)
	\qquad ( \bfalpha \in \T^r ),
\end{align}
when $a : \Z^d \rightarrow \C$ is a certain weight function.
We define the corresponding Gauss sum and oscillatory integral by
\begin{align*}
	S( \bfa,q )
	= \sum_{ \bfu \in \Z_q^d } e_q\bigg( \sum_{\bfj \in E} a_{\bfj} \bfu^{\bfj} \bigg)
	\quad (\bfa \in \Z_q^r),	
	\qquad
	I( \bfbeta )
	= \int_{ [0,1]^d } e\bigg( \sum_{\bfj \in E} \beta_{\bfj} \bfx^{\bfj} \bigg) \dbfx
	\quad (\bfbeta \in \R^r).
\end{align*}
%for $\bfalpha \in \T^r$, $q \geq 1$, $\bfa \in \Z_q^r$, $\bfbeta \in \R^r$.
By the multidimensional analogue of Hua's bound~\cite[Theorem~2.6]{ACK:Book}
and a standard van der Corput lemma~\cite[Corollary~2.3]{Parissis:Thesis},
we have
\begin{align}
	\label{eq:remov:MultiHuaBound}
	&\phantom{(q \geq 1, (\bfa,q) = 1)} &
	|S(\bfa,q)| 
	&\lesssim_\eps q^{d - 1/k + \eps}
	&&(q \geq 1, (\bfa,q) = 1),
	\\
	\label{eq:remov:MultiVdC}
	&\phantom{(\bfbeta \in \R^t)} &
	|I(\bfbeta)| 
	&\lesssim (1 + |\bfbeta|)^{-1/k}
	&&(\bfbeta \in \R^r).
\end{align}
For $p > 0$, define the local moments
\begin{align*}
	\frakS_p 
	= \sum_{q \geq 1} \sum_{\substack{ \bfa \in [q]^r \,: \\ (\bfa,q) = 1 }} |q^{-d} S(\bfa,q)|^p,
	\qquad	
	\frakJ_p 
	= \int_{\R^r} |I(\bfbeta)|^p \dbfbeta.
\end{align*}
By inserting the bounds~\eqref{eq:remov:MultiHuaBound}
and~\eqref{eq:remov:MultiVdC} in these expressions,
and using spherical coordinates to bound the second one,
we find that $\frakS_p < \infty$ for $p > k (r+1)$
and $\frakJ_p < \infty$ for $p > kr$.
Note also that estimates of Weyl type 
for the unweighted exponential sum
in~\eqref{eq:remov:WeylSumMultidim}
are available from early work of Arkhipov et al.~\cite[Theorem~3]{ACK:Weyl},
but for our purposes it is more expedient to quote the work of 
Parsell~\cite[Lemma~5.3, Theorem~5.5]{Parsell:MultidimVino}.
%(see also~\cite[Theorem~10.4]{PPW:Multidim}).
Using these ingredients as a replacement for 
Proposition~\ref{thm:remov:WeylSumBounds},
it is a straightforward deduction to obtain
the following multidimensional analogue
of Proposition~\ref{thm:remov:TruncRestr}.

\begin{proposition}[Truncated restriction estimate for monomial surfaces]
\label{thm:remov:TruncRestrMultidim}
Let $d \geq 1$ and let $E$ be a finite non-empty 
subset of $\N_0^d \smallsetminus \{ 0 \}$.
Consider the system of polynomials $\bfP = (\bfx^{\bfj},\,\bfj \in E)$
of dimension $d$, rank $r = |E|$, degree $k = \max_{\bfj \in E} |\bfj|$ 
and weight $K = \sum_{\bfj \in E} |\bfj|$.
There exists $\theta = \theta(d,r,k) > 0$ such that,
for $p > 2k(r+1)$,
\begin{align*}
	\int_{ |F_a^{(\bfP)}| \geq N^{d/2-\theta} \|a\|_2 } |F_a^{(\bfP)}|^p \ \dm 
	\lesssim_p N^{\tfrac{p}{2} - K} \| a \|_2^p.
\end{align*}
\end{proposition}

With a few more linear algebraic considerations
it is possible to obtain an absolutely analogous result
for general translation-dilation invariant systems
(where $d,r,k,K$ retain their usual meaning),
and we choose not to elaborate further on this point,
which does not require any essentially new idea.
Note that the above proposition misses the complete
supercritical range $p > 2K/d$, but it suffices
for our applications given the state of knowledge~\cite{PPW:Multidim} 
on multidimensional Vinogradov mean values.

\section{Additive equations of large degree}
\label{sec:large}

In this section we derive
Theorems~\ref{thm:intro:SystMonomMultidim},~\ref{thm:intro:SystMonom}
and~\ref{thm:intro:SystPols} on systems of equations of large degree.
We start by establishing a few simple facts about
translation-dilation invariant systems of polynomials.

\begin{lemma}
\label{thm:large:InjAndIneq}
Suppose that $\bfP$ is a translation-dilation invariant
system of $r$ polynomials of dimension $d$ and degree $k$.
Then $\bfx \mapsto \bfP(\bfx)$ is injective and $r \geq k$.
\end{lemma}

\begin{proof}
We first show that $k \leq r$.
%for translation-dilation invariant systems of polynomials.
Recall from~\cite[Section~2]{PPW:Multidim} that $\bfP = (P_1,\dots,P_r)$
is a translation-dilation invariant system
when the polynomials $P_1,\dots,P_r$ are homogeneous
of degree $k_i \geq 1$, and when there exist
integer polynomials $c_{j \ell}(\bfxi)$ in $d$ variables
for $1 \leq j \leq r$, $0 \leq \ell < j$ such that 
\begin{align*}
%\label{eq:large:TslDilInvEq}
	&\phantom{(\bfx,\, \bfxi \in \Z^d)} &
	P_j(\bfx + \bfxi) - P_j(\bfx) &= c_{j0}(\bfxi) + \sum_{\ell=1}^{j-1} c_{j \ell}(\bfxi) P_\ell(\bfx)
	&&(\bfx,\, \bfxi \in \Z^d).
\end{align*}
%Both sides are polynomials in $\R[\bfx,\bfxi]$ equal
%on $Z^d \times \Z^d$, and therefore the identity holds over $\R$ as well.
Performing a Taylor expansion of the left-hand side at $\bfx$,
and choosing $\bfxi = \bfe_i$ for an index $i \in [d]$ 
such that $x_i$ appears in a monomial of highest degree of $P_j$,
we may ensure that the left-hand side 
is a polynomial of degree $k_j - 1$ in $\bfx$, while the right-hand side is
a linear combination of polynomials of degrees $0,k_1,\dots,k_{j-1}$.
Consequently, we obtain the recursive bounds
$k_1 \leq 1$ and $k_j \leq \max_{\ell < j} k_\ell + 1$ for $j \geq 2$, 
so that upon iterating we derive $k_j \leq j$ for $1 \leq j \leq r$, 
and in particular $k = \max k_j \leq r$ as desired.

Next, note that the system of equations
$\bfP(\bfx) - \bfP(\bfy) = 0$
in variables $\bfx,\bfy \in \Z^d$ is translation-invariant.
Consider two fixed integers $\bfx,\bfy \in \Z^d$
such that $\bfP(\bfx) = \bfP(\bfy)$.
Then we have $\bfP(\bfx + \bfxi) = \bfP(\bfy + \bfxi)$
for every $\bfxi \in \Z^d$, and therefore for every $\bfxi \in \R^d$
by considering polynomials in the variable $\bfxi$.
By Taylor expansion at $\bfx$ and $\bfy$, we find that
$\partial^{\bfalpha} P_j(\bfx) = \partial^{\bfalpha} P_j(\bfy)$ for every 
$\bfalpha \in \N_0^d$ and every $j \in [r]$.
Since we assumed that at least one polynomial $P_j$
involves the variable $x_i$ for each $i \in [d]$,
it follows that $\bfx = \bfy$.
\end{proof}

Using an interpolation argument of Parsell et al.~\cite[Section~11]{PPW:Multidim},
we also find that the number of subset-sum solutions is always negligible
when a bound of the correct order of magnitude is available
for the relevant unweighted exponential sum.

\begin{lemma}
\label{thm:large:TrivSols}
Let $s \geq 3$ and $\lambda_1,\dots,\lambda_s \in \Z \smallsetminus \{0\}$
be such that $\lambda_1 + \dotsb + \lambda_s = 0$.
Suppose that $\bfP$ is a translation-dilation
invariant system of $r$ polynomials of dimension $d$,
degree $k$ and weight $K$.
Suppose that, for an integer $s > 2K/d$,
\begin{align*}
	\| F^{(\bfP)} \|_s^s \lesssim_\eps N^{ds - K + \eps}.
\end{align*}
Then the number of subset-sum solutions $\bfx \in [N]^d$
to~\eqref{eq:intro:SystPols} is bounded up to a constant factor
by $N^{ds - K - c}$, where $c = c(s,r,d,k) > 0$.
\end{lemma}

\begin{proof}
By injectivity of $\bfP$ (Lemma~\ref{thm:large:InjAndIneq})
and orthogonality we have immediately $\| F^{(\bfP)} \|_2^2 = N^d$.
Consider now a partition $[s] = E_1 \bigsqcup \dots \bigsqcup E_\ell$
with $\ell \geq 2$ and $\sum_{i \in E_j} \lambda_i = 0$ for all $j \in [\ell]$.
Since the $\lambda_i$ are nonzero, we have $m_j = |E_j| \in [2,s)$ for every $j \in [\ell]$.
We write $\calN_{(E_i)}(N)$ for the number of solutions
$\bfn_i \in [N]^d$ to the equations $\sum_{i \in E_j} \lambda_i \bfP(\bfn_i) = 0$, $j \in [\ell]$.
By orthogonality, Hölder's inequality and $1$-periodicity, we have
\begin{align*}
	\calN_{(E_i)}(N)
	&= \textstyle \prod_{j=1}^\ell \int_{\T^r} \prod_{i \in E_j} F(\lambda_i \bfalpha) \dbfalpha
	\\
	&\leq \textstyle \prod_{j=1}^\ell \prod_{i \in E_j} 
	\big[ \int_{\T^r} |F(\lambda_i \bfalpha)|^{m_j} \dbfalpha \big]^{\tfrac{1}{m_j}}
	\\
	&= \textstyle \prod_{j=1}^\ell \| F \|_{m_j}^{m_j}.
\end{align*}
Interpolating between $L^s$ and $L^2$, 
and observing that $\sum_{j=1}^\ell m_j = s$, we deduce that
\begin{align*}
	\calN_{(E_i)}(N)
	&\leq \textstyle \prod_{j=1}^\ell \big( \| F^{(\bfP)} \|_s^s \big)^{\tfrac{m_j-2}{s-2}} 
	\big( \| F^{(\bfP)} \|_2^2 \big)^{\tfrac{s-m_j}{s-2}}
	\\
	&\lesssim_\eps (N^{ds-K+\eps})^{\tfrac{s-2\ell}{s-2}} (N^d)^{\tfrac{\ell s - s}{s-2}}
	\\
	&= ( N^{ds - K + \eps} )^{1 - \tfrac{2(\ell-1)}{s-2}} (N^d)^{\tfrac{s(\ell-1)}{s-2}}.
\end{align*}
With further rearranging, we obtain
\begin{align*}
	\calN_{(E_i)}(N) 
	&\lesssim_\eps N^{ds - K + \eps} ( N^{2K - ds - 2\eps} )^{\tfrac{\ell-1}{s-2}}.
\end{align*}
Since $\ell \geq 2$, this last term is at most $O(N^{ds-K-c})$ 
for a certain $c = c(s,r,d,k) > 0$ when $s > 2K/d$, which is precisely our assumption.
\end{proof}

With these preliminaries in place, and from the
results of Sections~\ref{sec:add} and~\ref{sec:remov},
we can recover the theorems of the introduction on systems of large degree.
\smallskip

\textit{Proof of Theorem~\ref{thm:intro:SystMonom}.}
We want to apply Theorem~\ref{thm:intro:SystTslInv}
with $\bfP = (x,\dots,x^k)$ and $Z$ defined as the set
of projected or subset-sum solutions to~\eqref{eq:intro:SystMonom}.
We write $F = F^{(x,\dots,x^k)}$ and $F_a = F_a^{(x,\dots,x^k)}$,
and we let $\calN(N)$ denote the number of solutions 
$n_1,\dots,n_s \in [N]$ to~\eqref{eq:intro:SystMonom}.
%to the system of equations
%\begin{align}
%\label{eq:large:SystMonom}
%	&\phantom{(1 \leq j \leq k)} &
%	\lambda_1 n_1^j + \dotsb + \lambda_s n_s^j &= 0
%	&&(1 \leq j \leq k).
%\end{align}
Via the circle method~\cite[Section 9]{Wooley:EffI},
and assuming the existence
of nonsingular real and $p$-adic solutions to~\eqref{eq:intro:SystMonom},
one can obtain an asympotic formula 
of the form $\calN(N) \sim \frakS \cdot \frakJ \cdot N^{s - K}$ 
for $k \geq 3$ and $s > 2s_k$,
for certain constants $\frakS > 0$ and $\frakJ > 0$.

On the other hand, the projected solutions to~\eqref{eq:intro:SystMonom} 
are those such that $n_1 = \dots = n_s$,
and there are at most $N = N^{s - K - (s-K-1)}$ such solutions,
where $s-K-1 \geq 1$ since we have assumed $s > 2s_k \geq 2K$.
By Lemma~\ref{thm:large:TrivSols} and the 
estimate $\| F \|_{2s}^{2s} \lesssim N^{2s -  K + \eps}$
for $s \geq 2s_k \geq 2K$, 
the number of subset-sum solutions is also $O(N^{s-K-c})$ for a certain $c=c(s,k) >0$.
Therefore the assumption~\eqref{eq:intro:NumberSolsBounds}
is satisfied for $s > 2s_k$.

Finally, the restriction estimate~\eqref{eq:intro:RestrEpsFullLinfty}
is valid for any $s'' \geq 2s_k$,
via the bound
\begin{align*}
	&\phantom{(s \in \N)} &
	\| F_a \|_{2s}^{2s} 
	&\leq \| F \|_{2s}^{2s} \| a \|_\infty^{2s}
	= J_{s,k}(N) \| a \|_\infty^{2s}
	&&(s \in \N).
\end{align*}
The estimate~\eqref{eq:intro:RestrTruncAgain}, on the other hand,
% MANUAL
holds for some $\theta > 0$ and any $s' > 2K + 4$,
by Proposition~\ref{thm:remov:TruncRestr}.
Therefore, the assumptions of Theorem~\ref{thm:intro:SystTslInv}
are satisfied for $s > \max(2K + 4,2s_k)$,
and indeed for $s > 2K + 4$
upon using the result $s_k = 2K$ from~\cite{BDG:VinoMeanValue}.
\qed
\smallskip

\textit{Proof of Theorems~\ref{thm:intro:SystMonomMultidim}
and~\ref{thm:intro:SystPols}.}
We start by proving the more general Theorem~\ref{thm:intro:SystPols},
again by verifying the assumptions of Theorem~\ref{thm:intro:SystTslInv}.
For $s > 2r(k+1)$ and $s > K + d^2$, 
the work of Parsell et al.~\cite[Section~11]{PPW:Multidim}
shows that the assumptions~\eqref{eq:intro:NumberSolsBounds}
hold with a constant $\omega = \omega(s,r,d,k)$
when $Z$ is defined as the set of projected solutions
or subset-sum solutions to~\eqref{eq:intro:SystPols}
(one may instead use Lemma~\ref{thm:large:TrivSols}
and~\cite[Theorem~2.1]{PPW:Multidim} to bound the number of subset-sum solutions).

Assumption~\eqref{eq:intro:RestrEpsFullLinfty} holds for $s'' > 2r(k+1)$
by~\cite[Theorem~2.1]{PPW:Multidim} and using once more the inequality
\begin{align*}
	&\phantom{(s \in \N)} &
	\| F_a^{(\bfP)} \|_{2s}^{2s} 
	&\leq \| F^{(\bfP)} \|_{2s}^{2s} \| a \|_\infty^{2s}
	= J_s(N,\bfP) \| a \|_\infty
	&&(s \in \N).	
\end{align*}
The truncated restriction estimate~\eqref{eq:intro:RestrTruncAgain} 
holds for $s' > 2k(r+1)$ by the natural generalization of 
Proposition~\ref{thm:remov:TruncRestrMultidim}
to arbitrary reduced translation-dilation invariant systems $\bfP$, 
which we chose not to state.
Since $r \geq k$ by Lemma~\ref{thm:large:InjAndIneq}, 
we have $2r(k+1) \geq 2k(r+1)$, and therefore
this does not impose any additional constraint.
After choosing $\max(2r(k+1),K^2+d) < s'' < s' <s$, 
Theorem~\ref{thm:intro:SystTslInv} applies and
gives the desired conclusion.
In the special case $\bfP = (\bfx^{\bfj},\, 1 \leq |\bfj| \leq k)$,
it is explained in~\cite[Section~11]{PPW:Multidim}
that $2r(k+1) \geq K + d^2$, so that the assumption
$s > K + d^2$ becomes redundant, and Theorem~\ref{thm:intro:SystMonomMultidim} follows.
In that case the required estimate~\eqref{eq:intro:RestrTruncAgain} 
was explicitely stated as Proposition~\ref{thm:remov:TruncRestrMultidim},
taking $E = \{ \bfj \in \N_0^d \,:\, 1 \leq |\bfj| \leq k \}$.
\qed
\smallskip

We conclude this section with a small remark, which is that the
usual argument~\cite[Section~7]{Vaughan:Book} 
by which one obtains a lower bound of the correct
order of magnitude for $J_{s,k}(N)$ also shows that a system of equations
of the form~\eqref{eq:intro:SystPols} with symmetric
coefficients has the expected density of solutions 
in any subset of $[N]^d$.
This phenomenon was first observed by 
Rusza in the linear case~\cite[Theorem~3.2]{Ruzsa:Roth}.

\begin{proposition}
\label{thm:large:SystPolsSym}
Let $t \geq 1$ and
$\mu_1,\dots,\mu_t \in \Z \smallsetminus \{0\}$.
Suppose that $\bfP$ is a system of $r$ polynomials
having dimension $d$, degree $k$ and weight $K$.
Suppose that $A$ is a subset of $[N]^d$ of density $\delta$
and let $\calN(A,\bfP,\bfmu)$ denote the number of solutions 
$\bfn_i,\bfm_i \in A$ to the system of equations
\begin{align}
\label{eq:large:SystPolsSym}
	\mu_1 \bfP(\bfn_1) + \dotsb + \mu_t \bfP(\bfn_t)
	= 
	\mu_1 \bfP(\bfm_1) + \dotsb + \mu_t \bfP(\bfm_t)
\end{align}
in $s = 2t$ variables.
Then
\begin{align}
\label{eq:large:SystPolsSymLowerBound}
	\calN(A,\bfP,\bfmu) \gtrsim_{\bfP,\bfmu} \delta^s N^{ds - K}.
\end{align}
In particular, there exist constants $C(\bfP,\bfmu) > 0$
and $c(s,r,d,k) > 0$ such that
if $\delta \geq C(\bfP,\bfmu) N^{-c(s,r,d,k)}$, then $A$
contains a solution to~\eqref{eq:large:SystPolsSym},
which is neither a projected nor a subset-sum solution, 
provided also that
\begin{itemize}
	\item $\bfP = (x,\dots,x^k)$ and $s \geq 2s_k + 2$, or
	\item $\bfP = (\bfx^{\bfj},\, 1 \leq |\bfj| \leq k)$ and $s \geq 2r(k+1) + 2$, or
	\item $\bfP$ is an arbitrary system of polynomials and $s \geq \max(2r(k+1),K^2 + d) + 2$.
\end{itemize}
\end{proposition}

\begin{proof}
We write $\bfP = (P_1,\dots,P_r)$ and $k_i = \deg P_i$.
%and we let implicit constants depend on $\bfmu$ and $\bfP$ throughout.
For a set $E \subset \R^r$ and $\gamma \in \R$, we write
$\gamma \cdot E = \{ \gamma x, \, x \in E \}$,
and we also use traditional sumset notation in the proof.
We define $\bfP(A) = \{ \bfP(\bfn),\, \bfn \in A \}$
and a number-of-representations function
\begin{align*}
	&\phantom{(\bfu \in \Z^k)} &
	R( \bfu )
	&= \#\{\, \bfn_1,\dots,\bfn_t \in A \,:\, \mu_1 \bfP(\bfn_1) + \dotsb + \mu_t \bfP(\bfn_t) = \bfu \,\}
	&& (\bfu \in \Z^r).
\end{align*}
Summing over all $\bfu \in \Z^r$, we obtain
\begin{align*}
	|A|^t
	= \sum_{\bfu \,\in\, \mu_1 \cdot \bfP(A) + \dotsb + \mu_t \cdot \bfP(A)} R(\bfu).
\end{align*}
By Cauchy-Schwarz, it follows that
\begin{align*}
	|A|^{2t} \leq |\mu_1 \cdot \bfP(A) + \dotsb + \mu_t \cdot \bfP(A)| \cdot \sum_{\bfu \in \Z^r} R(\bfu)^2.
\end{align*}
Observing that
\begin{align*}
	\mu_1 \cdot \bfP(A) + \dotsb + \mu_t \cdot \bfP(A) 
	\subset \big[\! - O(N^{k_1}), O(N^{k_1}) \big] \times \dotsb \times \big[\! - O(N^{k_r}) , O(N^{k_r}) \big],
\end{align*}
where the implicit constants depend on $\bfP$ and $\bfmu$,
we have therefore
\begin{align*}
	\delta^{2t} N^{2dt} \lesssim_{\bfP,\bfmu} N^K \cdot \calN(A,\bfP,\bfmu).
\end{align*}
We recover~\eqref{eq:large:SystPolsSymLowerBound} after some rearranging.

In the various cases stated at the end of the proposition,
we have seen previously in this section that the number of 
projected or subset-sum solutions is 
$O_{\bfP,\bfmu}(N^{ds-K-c(s,r,d,k)})$
for some constant $c(s,r,d,k) > 0$, and therefore
we obtain solutions which are not of this kind for 
$\delta \geq C(\bfP,\bfmu) N^{-c'(s,r,d,k)}$,
for some $C(\bfP,\bfmu) > 0$ and $c'(s,r,d,k) > 0$.
\end{proof}

\section{The parabola system}
\label{sec:parab}

Fix $d \geq 1$, $s \geq 3$ and coefficients 
$\lambda_1,\dots,\lambda_s \in \Z \smallsetminus \{0\}$,
not necessarily summing up to zero.
%such that $\lambda_1 + \dotsb + \lambda_s = 0$.
We let $\calN(N,\bflambda)$
denote the number of solutions $\bfx_i \in [N]^d$
to the system of equations
\begin{align}
\label{eq:parab:SystParab}
\begin{split}
	\lambda_1 \bfx_1 + \dotsb + \lambda_s \bfx_s &= 0,	\\
	\lambda_1 |\bfx_1|^2 + \dotsb + \lambda_s |\bfx_s|^2 &= 0,
\end{split}
\end{align}
where $| \cdot |$ denote the Euclidean norm on $\R^d$.
This corresponds to the reduced translation-dilation invariant
system of polynomials 
$\bfP = (x_1,\dots,x_d,x_1^2 + \dotsb + x_d^2)$
of dimension $d$, rank $d+1$, degree $2$ and weight $d+2$.
We first observe that
$\calN(N,\bflambda)$ 
%and $\wt{\calN}(N,\bflambda)$
can be easily bounded from below
by inserting the linear equation into the quadratic one,
and invoking classical results on diagonal quadratic forms 
of rank at least five.

\begin{proposition}
\label{thm:parab:LowerBound}
%If the system of equations~\eqref{eq:parab:SystParab}
%has a nonzero real solution and $s \geq 1 + \tfrac{5}{d}$,
%we have
%\begin{align*}
%	\wh{\calN}(N,\bflambda) \gtrsim N^{ds - (d+2)}.
%\end{align*}
Suppose that $\lambda_1 + \dotsb + \lambda_s = 0$ and at least two
of the $\lambda_i$ are positive and at least two are negative,
and $s \geq \max(4,2+\tfrac{5}{d})$.
Then
\begin{align*}
	\calN(N,\bflambda) \gtrsim N^{ds - (d+2)}.
\end{align*}
\end{proposition}

\begin{proof}
We rewrite~\eqref{eq:parab:SystParab} as
\begin{align}
	\label{eq:parab:QuadForm}	
	\bfx_s = - \frac{1}{\lambda_s} \Big( \sum_{j=1}^{s-1} \lambda_j \bfx_j  \Big),
	\qquad
	\sum_{j=1}^{s-1} \lambda_s \lambda_j |\bfx_j|^2
	+ \bigg| \sum_{j=1}^{s-1} \lambda_j \bfx_j \bigg|^2 = 0.
\end{align}
We only consider solutions $(\bfx_i)$ with $\bfx_s$ as above and
$\bfx_i = \lambda_s \bfy_i$ for $1 \leq i < s$, with $\bfy_i \in [-cN,cN]^d$
for a small enough constant $c = c(\bflambda) > 0$.
By translation-invariance of~\eqref{eq:parab:SystParab},
such solutions may be shifted to
fit in the box $[N]^d$.
Unfolding the squared norm in the right-hand side 
of~\eqref{eq:parab:QuadForm},
we obtain a quadratic equation
\begin{align}
	\notag
	&&&&
	\sum_{j=1}^{s-1} \lambda_s \lambda_j |\bfy_j|^2
	+ \sum_{j,k \in [s-1]} \lambda_j \lambda_k \bfy_j \cdot \bfy_k &= 0
	\\
	\notag
	&& &\Leftrightarrow &
	\sum_{i=1}^d \bigg[
	\sum_{j=1}^{s-1} \lambda_s \lambda_j y_{ij}^2
	+ \sum_{j,k \in [s-1]} \lambda_j \lambda_{k} y_{ij} y_{ik} \bigg] 
	&= 0  
	\\
	\label{eq:parab:QuadEq}
	&& &\Leftrightarrow &
	 \wt{\bfy}^\transp B \wt{\bfy} &= 0,
\end{align}
where $\wt{\bfy} = [ \, [y_{1j}]_{j \in [s-1]} \, \dots \, [y_{dj}]_{j \in [s-1]} \, ]^\transp$ and
\begin{align*}
	B = 
	\begin{bmatrix} 
	A & & \\ & \ddots & \\ & & A  
	\end{bmatrix}
	\in \Z^{d(s-1) \times d(s-1)},
	\qquad
	A = [ \lambda_j ( \lambda_k + \delta_{jk} \lambda_s ) ]_{j,k \in [s-1]} \in \Z^{(s-1) \times (s-1)}.
\end{align*}

Under our assumptions on the $\lambda_i$,
it is established in the proof of~\cite[Proposition~7.3]{me:logkeil} 
that the quadratic form $\bfz \mapsto \bfz^\transp A \bfz$ 
is indefinite of rank $s-2$, 
and therefore $\wt{\bfy} \mapsto \wt{\bfy}^\transp B \wt{\bfy}$ 
is an indefinite quadratic form
in $d(s-1)$ variables of rank $d(s-2) \geq 5$
for $s \geq 2 + \frac{5}{d}$.
By diagonalizing $B$ and invoking
classical results on diagonal quadratic forms~\cite[Chapter~8]{Davenport:Book},
we find $\gtrsim N^{d(s-1)-2} = N^{ds-(d+2)}$ 
solutions $\bfy \in [-cN,cN]^{d(s-1)}$
to~\eqref{eq:parab:QuadEq}, 
and there are at least as many solutions 
$\bfx \in [N]^{ds}$ to 
the original system~\eqref{eq:parab:SystParab}.
\end{proof}

\begin{remark}
\label{thm:parab:LowerBoundRk}
Via the same method,
one can show that when $\sum_{i=1}^s \lambda_i \neq 0$,
the number of solutions to~\eqref{eq:parab:SystParab} in $[-N,N]^d \cap \Z^d$
is at least $c N^{ds-(d+2)}$, as long as $s \geq 1 + \tfrac{5}{d}$ and
there exists a nonzero real solution to~\eqref{eq:parab:SystParab}.
We do not insist on this point since we have opted
to work with quadrants $[N]^d$ throughout the article.
\end{remark}

Let us quote a crucial restriction estimate that
will be used in this section.

\begin{theorem}[Bourgain~\cite{Bourgain:ParabI}, Bourgain-Demeter~\cite{BD:DecouplConj}]
\label{thm:parab:RestrParab}
Suppose that $d \geq 1$ and
$\bfP = (x_1,\dots,x_d, x_1^2 + \dotsb + x_d^2)$. 
Then the estimates~\eqref{eq:intro:RestrCrit} 
and~\eqref{eq:intro:RestrEpsFree} hold
respectively for $p = 2(d+2)/d$ and $p > 2(d+2)/d$.
\end{theorem}

We also define an unweighted exponential sum
\begin{align}
\label{eq:parab:FDef}
%	&\phantom{( (\alpha,\bftheta) \in \T^{d+1} )} &
	F(\alpha,\bftheta) 
	= F^{(\bfP)}(\alpha,\bftheta)
	&=
	\sum_{ \bfn \in [N]^d } e( \alpha |\bfn|^2 + \bftheta \cdot \bfn )
	&&( (\alpha,\bftheta) \in \T^{d+1} )
\end{align}
associated to the $(d+1)$-dimensional parabola.
The estimate
\begin{align}
	\label{eq:parab:FMomentEpsFull}
	\| F \|_p^p \lesssim_p N^{dp - (d+2) + \eps}
	\quad\text{for $p \geq p_d = 2 + \frac{4}{d}$},
%	\\
%	\label{eq:parab:FMomentEpsFree}
%	\| F \|_p^p &\lesssim_p N^{dp - (d+2)}
%	\quad\text{for $p > p_d = 2 + \frac{4}{d}$},
\end{align}
which follows from
Theorem~\ref{thm:parab:RestrParab},
will be used in a few places.
It can be proven in a simpler way by the method
of Hu and Li~\cite[Theorem~1.3]{HuLi:Degree2}.
%: it will allow
%us to neglect the contribution of minor arcs in
%the count of solutions

First, we turn our attention to the problem of bounding
the number of trivial solutions,
and we need a complement to Proposition~\ref{thm:large:TrivSols}.
For distinct indices $i,j \in [s]$,
we let $\calN_{i,j}(N,\bflambda)$
denote the number of solutions $\bfx_1,\dots,\bfx_s \in [N]^d$
to~\eqref{eq:parab:SystParab} with $\bfx_i = \bfx_j$.

\begin{proposition}
\label{thm:parab:TrivSols}
For $s \geq \max( 4 , 2 + \frac{4}{d})$, there exists $c = c(d,s) > 0$
such that, for every pair of distinct indices $i,j \in [d]$,
\begin{align*}
	\calN_{i,j}(N,\bflambda) &\lesssim N^{ds - (d+2) - c}.
%	&&(, i,j \in [s] \,:\, i \neq j \,),
%	\\
%	\calN_{(E_i)}(N,\bflambda) &\lesssim N^{ds-(d+2)-c}
%	&&\textstyle (\, [s] = \bigsqcup_{i=1}^\ell E_i,\, \ell \geq 2,\, \sum_{i \in E_j} \lambda_i = 0 \,).
\end{align*}
\end{proposition}

\begin{proof}
We first show that, for a certain $c(t,s,d) > 0$,
\begin{align}
	\label{eq:parab:MomentInterpol}
	\| F \|_t^t \lesssim N^{s - (d+2) - c(t,s,d)}
	\quad\text{for}\quad 
	2 \leq t < s.
\end{align}
Indeed, by interpolation between $L^2$ and $L^s$,
and via~\eqref{eq:parab:FMomentEpsFull}, we obtain
\begin{align*}
	\| F \|_t^t
	&\leq ( \| F \|_s^s )^{1-\tfrac{s-t}{s-2}} ( \| F \|_2^2 )^{\tfrac{s-t}{s-2}} \\
	&\lesssim (N^{ds-(d+2) + \eps})^{1-\tfrac{s-t}{s-2}} (N^d)^{\tfrac{s-t}{s-2} } \\
	&\lesssim N^{ ds - (d+2) + \eps} (N^{2 - (s-2)d - \eps})^{ \tfrac{s-t}{s-2} },
\end{align*}
which is $\lesssim N^{ ds - (d+2) - c(t,s,d) }$
since $s > 2 + \tfrac{2}{d}$.
%Another fact we intend to use is that
%$\| F(\lambda \,\cdot \,) \|_p = \| F \|_p$
%for any $\lambda \in \Z \smallsetminus \{0\}$,
%which follows from $1$-periodicity.

Next, note that for distinct indices $i,j \in [s]$,
we have $\calN_{i,j}(N,\bflambda) \leq \calN(N,\bfmu)$
with $\bfmu \in (\Z \smallsetminus \{0\})^t$ and $t = s-1$ or $t = s-2$
according to whether $\lambda_i + \lambda_j = 0$ or not.
Observe also that
\begin{align*}
	\calN_{i,j}(N,\bfmu)
	= \int_{\T^{d+1}} F(\mu_1 \bfalpha) \cdots F(\mu_t \bfalpha) \dbfalpha 
	\leq \| F \|_t^t.
\end{align*}
We have $s-1 \geq s-2 \geq 2$ for $s \geq 4$,
and by~\eqref{eq:parab:MomentInterpol} it follows that
$\calN_{i,j}(N,\bfmu) \lesssim N^{ds-(d+2)-c(d,s)}$
for a certain $c(d,s) > 0$.
\end{proof}

At this stage we have developed enough machinery to
solve the system of equations~\eqref{eq:parab:SystParab}
in a thin subset of $[N]^d$.
\smallskip

\textit{Proof of Theorem~\ref{thm:intro:SystParab}.}
We wish to apply again Theorem~\ref{thm:intro:SystTslInv}.
%Note in advance that the conditions
%$s > 2 + \frac{4}{d}$ and $s \geq 2 + \frac{5}{d}$ 
%are equivalent for an integer $s$.
The bounds~\eqref{eq:intro:NumberSolsBounds}
are provided by Propositions~\ref{thm:parab:LowerBound}
and~\ref{thm:parab:TrivSols}
as well as Lemma~\ref{thm:large:TrivSols}
(which is applicable thanks to~\eqref{eq:parab:FMomentEpsFull}),
provided that $s \geq \max(4,2+\tfrac{5}{d})$,
a condition equivalent to the one stated in the theorem.
The full $L^2 \rightarrow L^p$ estimate of
Theorem~\ref{thm:parab:RestrParab} implies
of course~\eqref{eq:intro:RestrEpsFullLinfty}
and~\eqref{eq:intro:RestrTruncAgain} for
some real numbers $s',s''$ with $p_d = 2 + \tfrac{4}{d} < s'' < s' < s$. 
\qed
\smallskip

\begin{remark}
For $\bfP = (x_1,\dots,x_d,x_1^2 + \dotsb + x_d^2)$,
%\begin{align*}
%	F_a(\alpha,\bftheta) = \sum_{\bfn \in [N]^d} a(\bfn) e(\alpha |\bfn|^2 + \bftheta \cdot \bfn).
%\end{align*}
Bourgain~\cite[Propositions~3.6,~3.110 and~3.114]{Bourgain:ParabI} proved that
\begin{align*}
	\| F_a^{(\bfP)} \|_p^p \lesssim N^{dp/2 - (d+2)} \| a \|_2^p
\end{align*}
when $d=1$ and $p > 6$, 
or $d \geq 2$ and $p > 4$,
or $d \geq 4$ and $p \geq 2 + \tfrac{8}{d}$.
This can be used to obtain the conclusion of
Theorem~\ref{thm:intro:SystParab}
respectively for $d = 1$ and $s \geq 7$,
or $d \geq 2$ and $s \geq 5$, 
or $d \geq 5$ and $s \geq 4$.
\end{remark}

\medskip

In the second part of this section, we apply 
a traditional blend of the circle
method to derive an asymptotic formula for $\calN(N,\bflambda)$.
The bound~\eqref{eq:parab:FMomentEpsFull}
allows us to control the contribution of minor arcs,
and therefore most of our attention is devoted to the
major arc piece.
We define the Weyl sum
\begin{align*}
	&\phantom{((\alpha,\theta) \in \T^2)} &
	G(\alpha,\theta) &=
	\sum_{ n \in [N] } e( \alpha n^2 + \theta n )
	&&((\alpha,\theta) \in \T^2),
\end{align*}
so that by~\eqref{eq:parab:FDef} and
splitting of variables, we have
\begin{align}
\label{eq:parab:Fsplitting}
	F(\alpha,\bftheta) = \prod_{j=1}^d G(\alpha,\theta_j).
\end{align}
We also define a Gaussian sum and an oscillatory integral respectively by
\begin{align*}
	S(a,b;q) &= \sum_{u \bmod q} e_q( au^2 + bu )
	&&(q \geq 1,\, a,b \in \Z_q),
	\\
	I(\beta,\xi;N) &= \int_0^N e(\beta x^2 + \xi x) \dx
	&& (\beta,\xi \in \R),
\end{align*}
and we write $I(\beta,\xi)= I(\beta,\xi;1)$.
By a change of variables, we have
\begin{align}
\label{eq:parab:IRescaling}
	&\phantom{(\beta,\xi \in \R)} &
	I(\beta,\xi;N) &= N \cdot I(N^2\beta,N\xi)
	&&(\beta,\xi \in \R).
\end{align}
For a parameter $Q \geq 1$,
we define individual major arcs of level $Q$ by
\begin{align*}
%	\label{eq:parab:MajorArcs}
	&\phantom{= .} \frakM_Q(a,\bfb;q)
	\\
	\notag
	&= \{\, (\alpha,\bftheta) \in \T^{d+1} \,:\,
	\| \alpha - a/q \| \leq QN^{-2},\, \| \theta_j - b_j / q \| \leq QN^{-1} \ (1 \leq j \leq d) \,\},
\end{align*}
for any $q \geq 1$ and $(a,\bfb) \in [q]^{d+1}$.
We define the major and minor arcs of level $Q$ by
\begin{align}
\label{eq:parab:AllMajorArcs}
	\frakM_Q = \bigsqcup_{q \geq 1}\, \bigsqcup_{\substack{ (a,\bfb) \in [q]^{d+1} \\ (a,\bfb,q) = 1 }}
	\frakM_Q(a,\bfb,q),
	\qquad
	\frakm_Q = \T^{d+1} \smallsetminus \frakM_Q,
\end{align}
where one can check the union is indeed disjoint
when $Q \leq \tfrac{1}{2} N^{1/3}$.
%which we assume henceforth.
When the need arises, we will work with 
the fundamental domain $\frakU = (N^{-1/2},1 + N^{-1/2}]^{d+1}$ of $\T^{d+1}$.
The reason for this choice is of course that,
for $Q \leq \tfrac{1}{2} N^{1/2}$, 
\begin{align*}
	(a,\bfb)/q + [QN^{-2},QN^{-2}] \times [QN^{-1},QN^{-1}]^d \subset \frakU
	\quad\text{for}\quad
	1 \leq q \leq Q,\,
	(a,\bfb) \in [q]^{d+1}.
\end{align*}

We start by deriving major and minor arc bounds for the exponential sum~\eqref{eq:parab:Fsplitting}.

\begin{proposition}
\label{thm:parab:FBounds}
Suppose that $N^{1/100} \leq Q \leq N^{1/3}$.
For every $1 \leq q \leq Q$, $(a,\bfb) \in [q]^{d+1}$,
and $(\alpha,\bftheta) \in \frakM_Q(a,\bfb,q) \cap \frakU$, we have
\begin{align*}
	F(\alpha,\bftheta)
	= \prod_{j=1}^d q^{-1} S(a,b_j;q) I( \alpha - a/q, \theta_j - b_j/q ; N) + O( Q^{-1/4} N^d ) .
\end{align*}
For $(\alpha,\bftheta) \in \frakm_Q$, we have
\begin{align*}
	|F(\alpha,\bftheta)|\lesssim Q^{-1/4} N^d.
%	\quad\text{if $(\alpha,\bftheta) \in \frakm_Q$}.
\end{align*}
\end{proposition}

\begin{proof}
By Dirichlet's principle, we may find
$1 \leq a \leq q \leq 2^6 N$ with $(a,q) = 1$ such that
$|\alpha - a/q| \leq 2^{-6} q^{-1} N^{-1} \leq q^{-2}$.
If $q > Q$, it follows by Weyl's inequality~\cite[Lemma~2.4]{Vaughan:Book} that
$|G(\alpha,\theta_j)| \lesssim_\eps Q^{-1/2} N^{1+\eps} \lesssim Q^{-1/4} N$ 
for all $j \in [d]$,
and therefore $|F(\alpha,\bftheta)| \lesssim Q^{-d/4} N^d$ 
by~\eqref{eq:parab:Fsplitting}.

Next, fix a parameter $\eta \in (0,1]$ 
whose value shall be determined shortly.
If $q \leq Q$ and there exists $j \in [d]$ such that 
$|G(\alpha,\theta_j)| \leq \eta N$,
then clearly $|F(\alpha,\bftheta)| \leq \eta N^d$
by~\eqref{eq:parab:Fsplitting}.

In the case where $q \leq Q$ and
$|G(\alpha,\theta_j)| \geq \eta N$ for all $j \in [d]$,
we show that $(\alpha,\bftheta) \in \frakM_Q$
for a certain value of $\eta$.
By a final coefficient lemma~\cite[Lemma~4.6]{Baker:Book},
and assuming that $Q^{1/2} \leq \eta N^{1-\eps}$
for some $\eps > 0$, we may find an integer $1 \leq t_j \leq 2^6$ 
for every $j \in [d]$ such that,
writing $q_j = t_j q$, we have
\begin{align*}
	q_j \lesssim_\eps \eta^{-2} N^\eps,
	\quad
	\| q_j \alpha \| \lesssim_\eps \eta^{-2} N^{-2+\eps},
	\quad
	\| q_j \theta_j \| \lesssim \eta^{-2} N^{-1+\eps}.
\end{align*}
We let $q_0 = [q_1,\dots,q_k]$, and since
we have $\| q_0 \gamma \| \leq (q_0/q_j) \| q_j \gamma \|$
for every $\gamma \in \T$ and $j$, we deduce that
\begin{align*}
	q_0 \lesssim_\eps \eta^{-2} N^\eps,
	\quad
	\| q_0 \alpha \| \lesssim_\eps \eta^{-2} N^{-2+\eps},
	\quad
	\| q_0 \theta_j \| \lesssim \eta^{-2} N^{-1+\eps}.	
\end{align*}
Finally, choose $\eta = Q^{- 1/2 - \eps_0}$
for an $\eps_0 \in (0,1]$,
so that for $N$ large and $\eps$ small
we have $(\alpha,\bftheta) \in \frakM_Q$.

Working now with $(\alpha,\bftheta) \in \frakU \cap \frakM_Q(a,\bfb,q)$,
with $q \leq Q$ and $(a,\bfb) \in [q]^{d+1}$,
we have $| \alpha - a/q | \leq Q N^{-1}$ and
$| \theta_j - b_j/q | \leq Q N^{-2}$ for all $j$.
By the usual approximation formula~\cite[Theorem~7.2]{Vaughan:Book},
it follows that
\begin{align*}
	G(\alpha,\theta_j)  
	= q^{-1} S(a,b_j;q)I(\alpha - a/q,\theta_j - b_j/q ; N) + O( Q^2 )
\end{align*}
for all $j \in [d]$,
and we have $Q^2 \leq Q^{-1/4} N$. 
Taking the product over $j \in [d]$,
we obtain the required approximation of $F$ on $\frakM_Q(a,\bfb,q)$,
again by~\eqref{eq:parab:Fsplitting}.
\end{proof}

We treat in advance certain local moments
that will arise in our analysis.

\begin{proposition}
\label{thm:parab:LocalMomentsCv}
For $p > 0$ and $i \in [s]$, let
\begin{align}
	\label{eq:parab:LocalSeriesMoment}
	\frakS_{i,p}
	&= \sum_{q \geq 1} \sum_{\substack{ (a,\bfb) \in [q]^{d+1} \,: \\  (a,\bfb,q) = 1 }} \prod_{j=1}^d 
	\big| q^{-1} S\big( \lambda_i (a,b_j) ; q \big) \big|^p,
	\\
	\label{eq:parab:LocalIntgMoment}
	\frakJ_{i,p}
	&= \int_{\R^{d+1}} \prod_{j=1}^d 
	\big| I\big( \lambda_i(\beta,\xi_j) \big) \big|^p \dbeta \dbfxi.
\end{align}
Then $\frakS_{i,p} < \infty$ for $p > 2 + \tfrac{4}{d}$
and $\frakJ_{i,p} < \infty$ for $p > 2 + \tfrac{2}{d}$.
\end{proposition}

\begin{proof}
By Lemma~\ref{thm:gauss:GaussLemma} and
writing $h = (a,q)$ and $\lambda = \lambda_1 \cdots \lambda_s$
in~\eqref{eq:parab:LocalSeriesMoment}, we obtain
\begin{align*}
	\frakS_{i,p}
	&\lesssim_{\lambda_i}  \sum_{q \geq 1} \, \sum_{\substack{ 1 \leq a,b_1,\dots,b_d \leq q \,: \\ (a,b_1,\dots,b_d,q) = 1 }}
	1_{ h | \lambda (b_1,\dots,b_d) } \, h^{dp/2} q^{-dp/2}
	\\
	&\lesssim_{\lambda_i} \sum_{q \geq 1} q^{d+1 - dp/2}
\end{align*}
since $h | \lambda (a,b_1,\dots,b_d,q)$ implies $|h| \leq |\lambda|$,
and the last sum is absolutely convergent precisely for $p > 2(d+2)/d$.

By the usual van der Corput estimate,
and integrating first in the variables $\xi_j$ 
in~\eqref{eq:parab:LocalIntgMoment}, we also have
\begin{align*}
	\frakJ_{i,p}
	\lesssim \int_{\R} \prod_{j=1}^d \bigg[ \int_{\R} (1 + |\beta| + |\xi_j|)^{-p/2} \dxi_j \bigg] \dbeta.
\end{align*}
Note that $\int_0^\infty (1+a+x)^{-p/2} dx \asymp_p (1+a)^{1-p/2}$
for $a \geq 0$ and $p > 2$, and therefore under this assumption we have
\begin{align*}
	\frakJ_{i,p}
	\lesssim \int_{\R} (1+|\beta|)^{d(1-p/2)} \dbeta.
\end{align*}
This last integral is absolutely convergent for $p > 2 + \tfrac{2}{d}$.
\end{proof}

We define the
singular series and singular integral 
truncated at the level $T \geq 1$ respectively by
\begin{align}
	\label{eq:parab:SgSeries}
	\frakS(T) &= 
	\sum_{q \leq T} \sum_{(a,\bfb,q) = 1}
	\prod_{i=1}^s \prod_{j=1}^d q^{-1} S\big( \lambda_i(a,b_j) ; q \big),
	\\
	\label{eq:parab:SgIntg}
	\frakJ(T) &= 
	\int_{[-T,T]^{d+1}}
	\prod_{i=1}^s \prod_{j=1}^d q^{-1} I\big( \lambda_i(\beta,\xi_j) \big) \dbeta \dbfxi,
\end{align}
and when those converge absolutely 
we write $\frakS = \frakS(+\infty)$
and $\frakJ = \frakJ(+\infty)$.
By Hölder's inequality applied to products over $i \in [s]$,
and by Proposition~\ref{thm:parab:LocalMomentsCv},
it follows that we have absolute convergence
in~\eqref{eq:parab:SgSeries} and~\eqref{eq:parab:SgIntg}
for $s > 2 + \frac{4}{d}$.
We now have all the moment bounds needed to carry out our
main estimation.

\begin{proposition}
\label{thm:parab:Asympt}
For $s > 2 + \frac{4}{d}$, we have
$\frakS, \frakJ \in [0,\infty)$ and
there exists $\nu > 0$ such that
\begin{align*}
	\calN(N,\bflambda) = \frakS \cdot \frakJ \cdot N^{ds - (d+2)}
	+ O(N^{ds - (d+2) - \nu}).
\end{align*}
\end{proposition}

\begin{proof}
Throughout the proof, we use the letter $\nu$
to denote a small positive constant whose value
may change from line to line, but which remains bounded away from zero
in terms of $d$ and $s$. The letter $\eps$ denotes a positive constant
which may be taken arbitrarily small, 
and whose value may also change from line to line.
We fix $Q = N^{1/4}$, although the precise value is unimportant.
For a measurable subset $E$ of $\T^{d+1}$,
we define the multilinear operator
\begin{align*}
	T_E(K_1,\dots,K_s) = \int_E K_1 \cdots K_s \ \dm
\end{align*}
acting on functions $K_i : \T^{d+1} \rightarrow \C$.
For $p_d = 2 + \frac{4}{d}$ and any $i \in [s]$, we will use the bound
\begin{align}
\label{eq:parab:TBound}
	|T_E(K_1,\dots,K_s)|
	\leq \bigg[ \| K_i \|_{L^\infty(E)}^{s-p_d} \| K_i \|_{p_d}^{p_d}
	\prod_{j \in [s] \smallsetminus \{i\}} \| K_j \|_s^s \bigg]^{\tfrac{1}{s}}
\end{align}
which follows from Hölder's and Young's inequalities.
We define $F_i = F(\lambda_i \, \cdot \,)$,
so that
\begin{align}
\label{eq:parab:MomentBeginning}
	\calN(N,\bflambda) = T_{\T^{d+1}}(F_1,\dots,F_s).
\end{align}
Note that 
for any $P \geq 1$ and any $\lambda \in \Z \smallsetminus \{0\}$,
$(\alpha,\bftheta) \in \frakM_P$
implies $\lambda (\alpha,\bftheta) \in \frakM_{|\lambda| P}$,
and therefore $\lambda_i (\alpha,\bftheta) \in \frakm_{Q}$
implies $(\alpha,\bftheta) \in \frakm_{Q/|\lambda_i|}$
for any $i \in [s]$.
By Proposition~\ref{thm:parab:FBounds},
we have therefore $|F_i| \lesssim Q^{-1/4} N^d$ 
for all $i \in [s]$ on $\frakm_Q$.
From~\eqref{eq:parab:TBound} and~\eqref{eq:parab:FMomentEpsFull},
it follows that
\begin{align}
	\notag
	|T_{\frakm_Q}(F_1,\dots,F_s)|
	&\lesssim \big[ (N^{d-1/16})^{s - p_d} N^{dp_d -(d+2) + \eps } (N^{ds - (d+2) + \eps})^{s-1} \big]^{1/s} \\
	\notag
	&\lesssim N^{\eps - (1/16)(1 - p_d/s)} N^{ds - (d+2)}
	\\
	\label{eq:parab:MomentMinorArc}
	&\lesssim N^{ds - (d+2) - \nu}.
\end{align}

We now evaluate $T_{\frakM_Q}(F_1,\dots,F_s)$,
by replacing the exponential sums $F_i$ with their
usual major arc approximation.
For $i \in [s]$, we define the function $V_i : \frakU \rightarrow \C$ by
\begin{align}
\label{eq:parab:ViDef}
	V_i (\alpha,\bftheta) = 
	\prod_{j=1}^d q^{-1} S\big( \lambda_i(a,b_j) ; q \big) I( \alpha - a/q, \theta_j - b_j/q ; N)
	\qquad
	\text{for  $(\alpha,\bftheta) \in \frakM_Q(a,\bfb;q)$},
\end{align}
for every $q \geq 1$ and $(a,\bfb) \in [q]^{d+1}$ such that $(a,\bfb,q) = 1$,
and we define $V_i = 0$ on $\frakm_Q$.
Via Proposition~\ref{thm:parab:LocalMomentsCv} 
and~\eqref{eq:parab:IRescaling},
it is a simple matter to check that
\begin{align*}
	\| V_i \|_p^p \lesssim N^{dp - (d+2)}
	\qquad\text{for $p > 2 + \tfrac{4}{d}$}.
\end{align*}

Observe that if $(\alpha,\bftheta) \in \frakM_Q(a,\bfb,q)$
then $\lambda_i(\alpha,\bftheta) \in \frakM_{|\lambda_i|Q}(\lambda_i a , \lambda_i \bfb,q)$
for any $i \in [s]$.
Therefore, by Proposition~\ref{thm:parab:FBounds},
we have $| F_i - V_i| \lesssim N^{d - 1/16}$ on $\frakM_Q$.
Expanding $F_i = V_i + (F_i - V_i)$ by multilinearity,
and using a minor variant of~\eqref{eq:parab:TBound}, it follows that
\begin{align}
	\notag
	&\phantom{= .} |T_{\frakM_Q}(F_1,\dots,F_s) - T_{\frakM_Q}(V_1,\dots,V_s)|
	\\
	\notag
	&\lesssim \max\limits_{i \in [s]}
	\bigg[ \| F_i - V_i \|_\infty^{s-p_d-\eps} \| F_i - V_i \|_{p_d+\eps}^{p_d+\eps}
	\prod_{j \in [s] \smallsetminus \{i\}} \max( \| F_j \|_s^s, \| V_j \|_s^s ) \bigg]^{1/s}
	\\
	\notag
	&\lesssim N^{\eps - (1/16)(1 - p_d/s)} N^{ds - (d+2)}
	\\
	\label{eq:parab:MomentMajorArcApprox}
	&\lesssim N^{ds-(d+2)-\nu}.
\end{align}
for $\eps$ small enough.
Recall~\eqref{eq:parab:ViDef} and~\eqref{eq:parab:IRescaling},
so that by integrating over the fundamental domain $\frakU$
and summing over all the major arcs in~\eqref{eq:parab:AllMajorArcs},
we obtain
\begin{align}
	\notag
	&\phantom{= .} T_{\frakM_Q}(V_1,\dots,V_s)
	\\
	\notag
	&= \sum_{q \leq Q} \sum_{(a,\bfb,q) = 1} \prod_{i=1}^s \prod_{j=1}^d q^{-1} S\big( \lambda_i(a,b_j) ; q \big)  
	\\
	\notag
	&\phantom{= \sum_{q \leq Q}}
	\int\limits_{ [-QN^{-2},QN^{-2}] } \int\limits_{ [-QN^{-1},QN^{-1}]^d } 
	\prod_{i=1}^s \prod_{j=1}^d q^{-1} N I\big( \lambda_i(N^2\beta,N\xi_j) \big) \dbeta \dbfxi 
	\\
	\label{eq:parab:MomentMajorArcAsympt}
	&= \frakS(Q) \cdot \frakJ(Q) \cdot N^{ds - (d+2)},
\end{align}
where we have operated a change of variables 
$\beta \leftarrow N^2 \beta$, $\bfxi \leftarrow N\bfxi$ 
in the last step.
From the discussion following the introduction of
the singular series~\eqref{eq:parab:SgSeries} and~\eqref{eq:parab:SgIntg},
it follows that for $p > 2 + \frac{4}{d}$, 
we have $\frakS, \frakJ < \infty$ and
\begin{align*}
	\frakS(Q) = \frakS + O(N^{-\nu}),
	\quad
	\frakJ(Q) = \frakJ + O(N^{-\nu}).
\end{align*}
Inserting this into~\eqref{eq:parab:MomentMajorArcAsympt},
and recalling~\eqref{eq:parab:MomentBeginning},~\eqref{eq:parab:MomentMinorArc}
and~\eqref{eq:parab:MomentMajorArcApprox}, we obtain finally
\begin{align*}
	\calN(N,\bflambda)
	&= T_{\frakm_Q}(F_1,\dots,F_s) + 
	( T_{\frakM_Q}(F_1,\dots,F_s) - T_{\frakM_Q}(V_1,\dots,V_s) )
	+ T_{\frakM_Q}(V_1,\dots,V_s) 
	\\
	&= \frakS \cdot \frakJ \cdot N^{ds - (d+2)} + O(N^{ds - (d+2) - \nu}).
\end{align*}
\end{proof}

\textit{Proof of Theorem~\ref{thm:intro:SystParabAsympt}.}
Starting from Proposition~\ref{thm:parab:Asympt},
it suffices to carry out a classical
analysis~\cite[Chapter~20]{IK:Book} of the singular series $\frakS$ and the singular integral $\frakJ$,
after which one would find that $\frakS > 0$ and $\frakJ > 0$
under the stated assumptions.
Justifying a remark of the introduction, we mention that
if we had worked with an exponential sum of the form~\eqref{eq:parab:FDef}
defined over $[-N,N]^d \cap \Z^d$ instead, we would have obtained
an asymptotic formula for the number of solutions 
to~\eqref{eq:parab:SystParab} in that larger box, 
and by Remark~\ref{thm:parab:LowerBoundRk}
we could deduce that the corresponding singular factor is positive
whenever a nonzero real solution to~\eqref{eq:parab:SystParab} is known.
\qed

\appendix

\section{A uniform bound on Gauss sums}
\label{sec:gauss}

Here we include the proof of a well-known
estimate that we could not locate precisely
in the literature.

\begin{lemma}
\label{thm:gauss:GaussLemma}
For $q \geq 1$ and $a,b \in \Z_q$,
let $S(a,b;q) = \sum_{u \bmod q} e_q(au^2 + bu)$.
Uniformly in $q,a,b$, we have
\begin{align*}
	|S(a,b;q)| \lesssim 1_{(a,q) | b}\, (a,q)^{1/2} q^{1/2}.
\end{align*}
\end{lemma}

\begin{proof}
We let $h = (a,q)$, $a' = a/h$, $q' = q/h$.
We have
\begin{align}
	\notag
	S(a,b;q)
	&= \sum_{x \bmod q} e_{q'}(a' x^2) e_q( b x)
	\\
	\label{eq:gauss:SumCgr}
	&= \sum_{u \bmod q'} e_{q'}(a' u^2)
	\sum_{\substack{ x \bmod q \,: \\ x \equiv u \bmod q' }} e_q( b x ).
\end{align}
Writing $x = u + q' y$ with $y \in \Z_h$, we find that
\begin{align*}
	\sum_{\substack{ x \bmod q \,: \\ x \equiv u \bmod q' }} e_q( b x )
	= 	e_q(b u) \sum_{y \bmod h} e_h( b y )
	= e_q(b u) \cdot h 1_{h | b}.
\end{align*}
Inserting this back into~\eqref{eq:gauss:SumCgr},
we find that $S(a,b;q) = 0$ if $h \nmid b$,
and else we write $b = h b'$ and obtain
\begin{align*}
	S(a,b;q)
	&= h \sum_{u \bmod q'} e_{q'}(a' u^2 + b' u).
\end{align*}
Since $(a',q') = 1$ and $q' = q/h$, 
the usual squaring-differencing argument then gives
\begin{align*}
	|S(a,b;q)| \lesssim h (q/h)^{1/2} = (hq)^{1/2}. 
\end{align*}
\end{proof}

\bibliographystyle{amsplain}
\bibliography{addeqs_arxiv5}

\providecommand{\bysame}{\leavevmode\hbox to3em{\hrulefill}\thinspace}
\providecommand{\MR}{\relax\ifhmode\unskip\space\fi MR }
\providecommand{\MRhref}[2]{%
  \href{http://www.ams.org/mathscinet-getitem?mr=#1}{#2}
}
\providecommand{\href}[2]{#2}
\begin{thebibliography}{10}

\bibitem{ACK:Book}
G.~I. Arkhipov, V.~N. Chubarikov, and A.~A. Karatsuba, \emph{Trigonometric sums
  in number theory and analysis}, Walter de Gruyter, Berlin, 2004.

\bibitem{ACK:Weyl}
G.~I. Arkhipov, A.~A. Karatsuba, and V.~N. Chubarikov, \emph{An upper bound on
  the modulus of a multiple trigonometric sum}, Trudy Mat. Inst. Steklov.
  \textbf{143} (1977), 3--31.

\bibitem{ACK:SgIntg}
\bysame, \emph{Trigonometric integrals}, Izv. Akad. Nauk SSSR Ser. Mat.
  \textbf{43} (1979), no.~5, 971--1003, 1197.

\bibitem{Baker:Book}
R.~C. Baker, \emph{Diophantine inequalities}, The Clarendon Press, Oxford
  University Press, New York, 1986.

\bibitem{Bourgain:VinoSurvey}
J.~Bourgain, \emph{On the {V}inogradov mean value}, Preprint (2016),
  \url{http://arxiv.org/abs/1601.08173}.

\bibitem{Bourgain:Squares}
\bysame, \emph{On {$\Lambda(p)$}-subsets of squares}, Israel J. Math.
  \textbf{67} (1989), no.~3, 291--311.

\bibitem{Bourgain:SphereI}
\bysame, \emph{Eigenfunction bounds for the {L}aplacian on the {$n$}-torus},
  Internat. Math. Res. Notices (1993), no.~3, 61--66.

\bibitem{Bourgain:ParabI}
\bysame, \emph{Fourier transform restriction phenomena for certain lattice
  subsets and applications to nonlinear evolution equations. {I}.
  {S}chr\"odinger equations}, Geom. Funct. Anal. \textbf{3} (1993), no.~2,
  107--156.

\bibitem{Bourgain:ParabII}
\bysame, \emph{Fourier transform restriction phenomena for certain lattice
  subsets and applications to nonlinear evolution equations. {II}. {T}he
  {K}d{V}-equation}, Geom. Funct. Anal. \textbf{3} (1993), no.~3, 209--262.

\bibitem{Bourgain:Roth}
\bysame, \emph{On triples in arithmetic progression}, Geom. Funct. Anal.
  \textbf{9} (1999), no.~5, 968--984.

\bibitem{BD:DecouplConj}
J.~Bourgain and C.~Demeter, \emph{The proof of the $\ell^2$ decoupling
  conjecture}, Ann. of Math. \textbf{182} (2015), no.~1, 351--389.

\bibitem{BDG:VinoMeanValue}
J.~Bourgain, C.~Demeter, and L.~Guth, \emph{Proof of the main conjecture in
  {V}inogradov's mean value theorem for degrees higher than three}, Preprint
  (2015), \url{http://arxiv.org/abs/1512.01565}.

\bibitem{Davenport:Book}
H.~Davenport, \emph{Analytic methods for {D}iophantine equations and
  {D}iophantine inequalities}, second ed., Cambridge University Press,
  Cambridge, 2005.

\bibitem{Drury:RestrCurves}
S.~W. Drury, \emph{Restrictions of {F}ourier transforms to curves}, Ann. Inst.
  Fourier (Grenoble) \textbf{35} (1985), no.~1, 117--123.

\bibitem{Galambos:Book}
J.~Galambos, \emph{Advanced probability theory}, second ed., Marcel Dekker,
  Inc., New York, 1995.

\bibitem{Green:RestrCourse}
B.~Green, \emph{Restriction and {K}akeya phenomena}, Lecture notes (2002),
  \url{http://people.maths.ox.ac.uk/greenbj/papers/rkp.pdf}.

\bibitem{Green:Factors}
\bysame, \emph{Montr\'eal notes on quadratic {F}ourier analysis}, Additive
  combinatorics, CRM Proc. Lecture Notes, vol.~43, Amer. Math. Soc.,
  Providence, RI, 2007, pp.~69--102.

\bibitem{GT:DiophApprox}
B.~Green and T.~Tao, \emph{New bounds for {S}zemer\'edi's theorem. {II}. {A}
  new bound for {$r_4(N)$}}, Analytic number theory, Cambridge Univ. Press,
  Cambridge, 2009, pp.~180--204.

\bibitem{HB:Roth}
D.~R. Heath-Brown, \emph{Integer sets containing no arithmetic progressions},
  J. London Math. Soc. (2) \textbf{35} (1987), no.~3, 385--394.

\bibitem{me:logkeil}
K.~Henriot, \emph{Logarithmic bounds for translation-invariant equations in
  squares}, To appear in Int. Math. Res. Nat. IMRN (2015).

\bibitem{HuLi:Degree2}
Yi~Hu and X.~Li, \emph{Discrete {F}ourier restriction associated with
  {S}chr\"odinger equations}, Rev. Mat. Iberoam. \textbf{30} (2014), no.~4,
  1281--1300.

\bibitem{Hua:SgSeriesIntg}
L.-K. Hua, \emph{On the number of solutions of {T}arry's problem}, Acta Sci.
  Sinica \textbf{1} (1952), 1--76.

\bibitem{Hughes:EpsRemoval}
K.~Hughes, \emph{Extending {B}ourgain's discrete epsilon-removal lemma},
  Unpublished (2013).

\bibitem{IK:Book}
H.~Iwaniec and E.~Kowalski, \emph{Analytic number theory}, American
  Mathematical Society, Providence, RI, 2004.

\bibitem{Keil:QuadFormsII}
E.~Keil, \emph{Some refinements for translation invariant quadratic forms in
  dense sets}, Preprint (2014), \url{http://arxiv.org/abs/1408.1535}.

\bibitem{Keil:QuadFormsI}
\bysame, \emph{Translation invariant quadratic forms in dense sets}, Preprint
  (2013), \url{http://arxiv.org/abs/1308.6680}.

\bibitem{Keil:Diag}
\bysame, \emph{On a diagonal quadric in dense variables}, Glasg. Math. J.
  \textbf{56} (2014), no.~3, 601--628.

\bibitem{LM:DiophApprox}
N.~Lyall and A.~Magyar, \emph{Simultaneous polynomial recurrence}, Bull. Lond.
  Math. Soc. \textbf{43} (2011), no.~4, 765--785.

\bibitem{MT:RestrFF}
G.~Mockenhaupt and T.~Tao, \emph{Restriction and {K}akeya phenomena for finite
  fields}, Duke Math. J. \textbf{121} (2004), no.~1, 35--74.

\bibitem{Parissis:Thesis}
I.~Parissis, \emph{Oscillatory integrals with polynomial phase}, Ph.D. thesis,
  University of Crete, 2007,
  \url{http://www.math.aalto.fi/~parissi1/PhD_thesis/thesis.pdf}.

\bibitem{Parsell:MultidimVino}
S.~T. Parsell, \emph{A generalization of {V}inogradov's mean value theorem},
  Proc. London Math. Soc. (3) \textbf{91} (2005), no.~1, 1--32.

\bibitem{PPW:Multidim}
S.~T. Parsell, S.~M. Prendiville, and T.~D. Wooley, \emph{Near-optimal mean
  value estimates for multidimensional {W}eyl sums}, Geom. Funct. Anal.
  \textbf{23} (2013), no.~6, 1962--2024.

\bibitem{Prendiville:BinaryForms}
S.~M. Prendiville, \emph{Solution-free sets for sums of binary forms}, Proc.
  Lond. Math. Soc. (3) \textbf{107} (2013), no.~2, 267--302.

\bibitem{Roth:Roth}
K.~F. Roth, \emph{On certain sets of integers}, J. London Math. Soc.
  \textbf{28} (1953), 104--109.

\bibitem{Ruzsa:Roth}
I.~Z. Ruzsa, \emph{Solving a linear equation in a set of integers. {I}}, Acta
  Arith. \textbf{65} (1993), no.~3, 259--282.

\bibitem{Smith:VinoSyst}
M.~L. Smith, \emph{On solution-free sets for simultaneous additive equations},
  Ph.D. thesis, University of Michigan, 2007,
  \url{http://dept.math.lsa.umich.edu/research/number_theory/theses/matthew_smith.pdf}.

\bibitem{Smith:Diag}
\bysame, \emph{On solution-free sets for simultaneous quadratic and linear
  equations}, J. Lond. Math. Soc. (2) \textbf{79} (2009), no.~2, 273--293.

\bibitem{Szemeredi:Roth}
E.~Szemer{\'e}di, \emph{Integer sets containing no arithmetic progressions},
  Acta Math. Hungar. \textbf{56} (1990), no.~1-2, 155--158.

\bibitem{Tao:RestrExpo}
T.~Tao, \emph{Some recent progress on the restriction conjecture}, Fourier
  analysis and convexity, Birkh\"auser Boston, Boston, MA, 2004, pp.~217--243.

\bibitem{Tao:Factors}
\bysame, \emph{A quantitative ergodic theory proof of {S}zemer\'edi's theorem},
  Electron. J. Combin. \textbf{13} (2006), no.~1.

\bibitem{Vaughan:Book}
R.~C. Vaughan, \emph{The {H}ardy-{L}ittlewood method}, second ed., Cambridge
  University Press, Cambridge, 1997.

\bibitem{Wolff:HA}
T.~H. Wolff, \emph{Lectures on harmonic analysis}, vol.~29, American
  Mathematical Society, Providence, RI, 2003, Edited by {\L}aba and Carol
  Shubin.

\bibitem{Wooley:Cubic}
T.~D. Wooley, \emph{The cubic case of the main conjecture in {V}inogradov's
  mean value theorem}, Preprint (2014), \url{http://arxiv.org/abs/1401.3150}.

\bibitem{Wooley:Restr}
\bysame, \emph{Discrete {F}ourier restriction via efficient congruencing: basic
  principles}, Submitted.

\bibitem{Wooley:Banff}
\bysame, \emph{Restriction theory and perturbation of {W}eyl sums}, Talk slides
  (July 2015),
  \url{https://www.birs.ca/workshops/2015/15w5013/files/wooley.pdf}.

\bibitem{Wooley:Clay}
\bysame, \emph{Vinogradov mean value theorem and restriction theory}, Talk
  slides (October 2014),
  \url{http://www.claymath.org/sites/default/files/2014oxfordclaytalk.pdf}.

\bibitem{Wooley:EffI}
\bysame, \emph{Vinogradov's mean value theorem via efficient congruencing},
  Ann. of Math. (2) \textbf{175} (2012), no.~3, 1575--1627.

\bibitem{Wooley:EffII}
\bysame, \emph{Vinogradov's mean value theorem via efficient congruencing,
  {II}}, Duke Math. J. \textbf{162} (2013), no.~4, 673--730.

\end{thebibliography}

\bigskip

\textsc{\footnotesize Department of mathematics,
University of British Columbia,
Room 121, 1984 Mathematics Road,
Vancouver BC V6T 1Z2, Canada
}

\textit{\small Email address: }\texttt{\small khenriot@math.ubc.ca}

\end{document}